\documentclass[amstex,12pt,reqno]{amsart}
\usepackage{amsmath, amssymb, amsthm}
\usepackage[dvipdfmx]{graphicx}
\usepackage{mathrsfs}
\usepackage[all]{xy}
\usepackage{tikz}
\usetikzlibrary{intersections,calc,arrows.meta,patterns}

\newtheorem{theorem}{Theorem}[section]
\newtheorem{lemma}[theorem]{Lemma}
\newtheorem{proposition}[theorem]{Proposition}
\newtheorem{corollary}[theorem]{Corollary}

\theoremstyle{definition}
\newtheorem*{definition}{Definition}
\newtheorem*{remark}{Remark}

\title[BMO embeddings, chord-arc curves, Riemann mapping parametrization]
{BMO embeddings, chord-arc curves, and \\Riemann mapping parametrization}

\author[H. Wei]{Huaying Wei} 
\address{Department of Mathematics and Statistics, Jiangsu Normal University \endgraf Xuzhou 221116, PR China} 
\curraddr{Department of Mathematics, School of Education, Waseda University \endgraf
Shinjuku, Tokyo 169-8050, Japan}
\email{hywei@jsnu.edu.cn} 

\author[K. Matsuzaki]{Katsuhiko Matsuzaki}
\address{Department of Mathematics, School of Education, Waseda University \endgraf
Shinjuku, Tokyo 169-8050, Japan}
\email{matsuzak@waseda.jp}

\makeatletter
\@namedef{subjclassname@2020}{%
\textup{2020} Mathematics Subject Classification}
\makeatother

\thanks{Research supported by 
Japan Society for the Promotion of Science (KAKENHI 18H01125 and 21F20027).}

\begin{document}

\maketitle

\begin{abstract}
We consider the space of chord-arc curves on the plane passing through the infinity with 
their parametrization $\gamma$ on the real line,
and embed this space into the product of the BMO Teichm\"uller spaces. The fundamental theorem we prove about
this representation is that $\log \gamma'$ also gives a biholomorphic homeomorphism 
into the complex Banach space of BMO functions. Using these two equivalent complex structures,
we develop a clear exposition on the analytic dependence of involved mappings between certain subspaces. Especially, we
examine the parametrization of a chord-arc curve by using the Riemann mapping and
its dependence on the arc-length parametrization. As a consequence, we can solve a conjecture of
Katznelson, Nag, and Sullivan in 1990
by showing that this dependence is not continuous.
\end{abstract}

\section{Introduction}

A quasicircle $\Gamma$ is the image of the real line $\mathbb R$ by a quasiconformal homeomorphism of
the complex plane $\mathbb C$. (In this paper, we always assume that a closed curve passes through the infinity.)
The family of all such quasicircles modulo affine translation is identified with the universal Teichm\"uller space $T$.
However, if we consider $\Gamma$ with its parametrization, namely, if we consider an embedding
$\mathbb R \to \mathbb C$ that is induced by a quasiconformal homeomorphism of $\mathbb C$,
the family of all such normalized (i.e., $0$, $1$, and $\infty$ are fixed) embeddings 
is identified with the product of the Teichm\"uller spaces $T(\mathbb U) \times T(\mathbb L)$ 
defined on the upper and the lower
half-planes. This representation has been used to investigate
quasifuchsian spaces by the Bers simultaneous uniformization.

We apply this method to consider chord-arc curves. A chord-arc curve $\Gamma$ is 
the image of $\mathbb R$ by a bi-Lipschitz homeomorphism of
$\mathbb C$. The family of all such chord-arc curves modulo affine translation
is identified with a proper subset of the BMO Teichm\"uller space $T_b$
introduced by Astala and Zinsmeister \cite{AZ}. To give a parametrization for $\Gamma$,
we introduce normalized BMO embeddings of $\mathbb R$ into $\mathbb C$ whose totality is identified with
the product $T_b(\mathbb U) \times T_b(\mathbb L)$. 
The subset of all normalized BMO embeddings whose images are chord-arc curves is denoted by $\rm CA$.
Every $\gamma \in {\rm CA}$ is locally absolutely continuous and $\log \gamma'$
belongs to the complex Banach space ${\rm BMO}(\mathbb R)$ of all complex-valued BMO functions on $\mathbb R$
modulo constant functions.
In this paper, we prove the following fundamental result on $\rm CA$ from the view point of quasiconformal Teichm\"uller theory
(see Theorem \ref{biholo}).
Differently from the case of quasicircles, the better regularity of chord-arc curves makes it possible 
for $\rm CA$ to have another representation of its complex structure in the complex Banach space ${\rm BMO}(\mathbb R)$.

\begin{theorem}\label{fundamental}
$\rm CA$ is an open subset of $T_b(\mathbb U) \times T_b(\mathbb L)$, and the map $L:{\rm CA} \to {\rm BMO}(\mathbb R)$ defined by
$L(\gamma)=\log \gamma'$ is a biholomorphic homeomorphism onto its image.
\end{theorem}

For the study of plane curves which have a nature of problems in harmonic analysis,
this point of view has missed until recently we begin to do researches for Weil--Petersson curves (see \cite{WM-4}). 
In this paper, we will show that the above theorem greatly simplify and clarify the arguments concerning
chord-arc curves by giving simple proofs for existing important results and
also by answering open problems in this subject matter.

A quasisymmetric homeomorphism $f:\mathbb R \to \mathbb R$ is the extension of a quasiconformal homeomorphism of
$\mathbb U$ onto itself. If $f$ is locally absolutely continuous and its derivative $f'$ is an $A_\infty$-weight
of Muckenhoupt, then $f$ is called strongly quasisymmetric. The group of 
all normalized strongly quasisymmetric homeomorphism on $\mathbb R$ is
denoted by ${\rm SQS}$, which can be identified with the BMO Teichm\"uller space $T_b$.
Let ${\rm BMO}_{\mathbb R}^*(\mathbb R)$ denote a convex open subset of 
the real Banach space of all real-valued BMO functions $u$ such that $e^u$ is an $A_\infty$-weight.
Then, for every $f \in {\rm SQS}$, $\log f'$ belongs to ${\rm BMO}_{\mathbb R}^*(\mathbb R)$ by definition,
and in fact, this correspondence $\Psi:{\rm SQS} \cong T_b \to {\rm BMO}_{\mathbb R}^*(\mathbb R)$ is bijective.
Moreover, it is known that $\Psi$ is a homeomorphism. 
However, in our formulation, we can apply a simpler argument to this map $\Psi$ to obtain a stronger result.
Since ${\rm SQS}$ is a real-analytic submanifold of ${\rm CA}$ corresponding to the diagonal locus of 
$T_b(\mathbb U) \times T_b(\mathbb L)$
and $\Psi$ is the restriction of $L$ to ${\rm SQS}$,
Theorem \ref{fundamental} implies the following assertion (see Corollary \ref{real-analytic}).
This gives an independent argument for an expected claim by Fan, Hu and Shen \cite[Remark 4.4]{FHS}.

\begin{corollary}\label{1.2}
$\Psi:{\rm SQS} \cong T_b \to {\rm BMO}_{\mathbb R}^*(\mathbb R)$ is a real-analytic homeomorphism whose
inverse $\Psi^{-1}$ is also real-analytic.
\end{corollary}

The importance of this result lies in a fact that we can convert the real analytic dependence upon the BMO norm
to that upon the analytic structure of the Teichm\"uller space. This is the translation of harmonic analysis aspects into
complex analysis aspects. For instance, the tangent space of ${\rm SQS}$ can be described by the solution of
a certain time dependent flow equation on $\mathbb R$ but the dependence of the solution should be given by the
BMO norm (see \cite{WS}). Since the tangent space of $T_b$ is represented by Beltrami differentials in the complex analytic theory,
we need a translation of the relation ${\rm SQS} \cong T_b$ in the level of tangent spaces.
Corollary \ref{1.2} is useful for this.

Next, we consider a problem on the dependence of Riemann mapping parametrization of
chord-arc curves. For each chord-arc curve $\Gamma$, the normalized Riemann mapping
from $\mathbb U$ to the left domain bounded by $\Gamma$ defines a parametrization of $\Gamma$ by
its extension to $\mathbb R$.
Let ${\rm RM}^\circ$ be a subset of ${\rm CA}$ consisting of all such Riemann mapping parametrizations of chord-arc curves.
This is a complex-analytic submanifold of $\rm CA$.
Let $\Phi:{\rm CA} \to {\rm RM}^\circ$ be defined by taking the Riemann mapping parametrization in the same chord-arc image. 
Associated with $\Phi$, another map $\Pi:{\rm CA} \to {\rm SQS}$ for $\gamma \in {\rm CA}$ 
is defined by the reparamatrization 
$\gamma$ from $\Phi(\gamma)$ by a strongly quasisymmetric homeomorphism $\Pi(\gamma)$.
Namely, there is a unique decomposition $\gamma=\Phi(\gamma) \circ \Pi(\gamma)$.
This map $\Pi$ can be regarded as the projection onto the real-analytic submanifold ${\rm SQS}$, and hence
it is real-analytic. 

On the other hand, since a chord-arc curve $\Gamma$ is rectifiable, it has the arc-length parametrization.
Let ${\rm ICA}$ be a subset of $\rm CA$ consisting of arc-length parametrizations of chord-arc curves.
This is the inverse image of some open subset $\Omega$ in
the real Banach subspace $i{\rm BMO}_{\mathbb R}(\mathbb R)$ of purely imaginary BMO functions   
by $L$, and hence ${\rm ICA}$ is a real-analytic submanifold of ${\rm CA}$.
The problems we will consider are concerning the maps $\Pi$ and $\Phi$ restricted to ${\rm ICA}$.
By the identification of ${\rm ICA}$ with $\Omega \subset i{\rm BMO}_{\mathbb R}(\mathbb R)$
(which is denoted by $i{\rm BMO}_{\mathbb R}(\mathbb R)^\circ$ later) 
and also by ${\rm SQS} \cong {\rm BMO}_{\mathbb R}^*(\mathbb R)$, 
we can define a map $\lambda:\Omega \to {\rm BMO}_{\mathbb R}^*(\mathbb R)$
as the conjugate of $\Pi|_{\rm ICA}$ by $L$. Properties of this $\lambda$ were intensively investigated and 
their arguments were developed further in the literature. Among them, one of the important results Coifman and Meyer \cite{CM}
obtained is that $\lambda$ is real-analytic. See also Semmes \cite{SeB} and Wu \cite{Wu}.
In our formulation, since $\Pi|_{\rm ICA}$ is merely the projection from
the real-analytic submanifold $\rm ICA$, we see that this is an immediate consequence from the fact that
$L$ is biholomorphic (see Theorem\ref{CM}).

\begin{corollary}\label{1.3}
$\lambda=L \circ \Pi|_{\rm ICA} \circ L^{-1}:\Omega \to {\rm BMO}_{\mathbb R}^*(\mathbb R)$ is real-analytic.
\end{corollary}
 
We also consider $\Phi|_{\rm ICA}$. This map indicates how
the Riemann mapping varies according to the change of the chord-arc curves $\Gamma$.
By the composition of
$\Pi^*$ identifying ${\rm RM}$ with the Teichm\"uller space $T_b \cong \rm SQS$,
we also obtain the correspondence from ${\rm ICA}$ to conformal welding homeomorphisms in $\rm SQS$.
The map $\rho:\Omega \to {\rm BMO}_{\mathbb R}^*(\mathbb R)$ given as the conjugate of $\Pi^* \circ \Phi|_{\rm ICA}$ by $L$
was considered in
Katznelson, Nag and Sullivan \cite{KNS} where they mentioned that $\rho$ was not known to be continuous, and
stated several preferable consequences ``if this were continuous''. On the contrary, Shen and Wu \cite{SW}
proved that the corresponding mapping to $\rho$ is continuous in a similar setting of Weil--Petersson curves. This can be also
explained simply from another fundamental theorem for the space of Weil--Petersson curves (see \cite{WM-4}).
Our conclusion for this problem in this paper is:

\begin{theorem}\label{1.4}
$\Phi|_{\rm ICA}:{\rm ICA} \to {\rm RM}^\circ$ and $\rho=L \circ (\Pi^* \circ \Phi|_{\rm ICA}) \circ L^{-1}:\Omega \to {\rm BMO}_{\mathbb R}^*(\mathbb R)$ are not continuous.
\end{theorem}

See Theorem \ref{main} and Corollary \ref{another}. 
This is proved by using a fact that $T_b \cong {\rm SQS}$ is not a topological group.
This property implies that $\Phi$ is not continuous on ${\rm CA}$,
but in order to apply this on ${\rm ICA}$, we have to investigate a local property of $\lambda$ at the origin.
By using a fact that the derivative of $\lambda$ at the origin is given by the Hilbert transformation on
${\rm BMO}_{\mathbb R}(\mathbb R)$, we see that $\lambda$ is a local homeomorphism at the origin.
This gives us the freedom of choosing elements in ${\rm ICA}$ that yields a particular example showing the discontinuity of $\Phi|_{\rm ICA}$.

\medskip
Here is a brief introduction of the contents in this paper. Section 2 is a collection of preliminary results used later. It contains
BMO functions, $A_\infty$-weights, Carleson measures, and their roles in the quasiconformal theory of Teichm\"uller spaces.
In Section 3, we introduce BMO embeddings of the real line 
in general and give a rather complete survey on the results
concerning such embeddings whose images are chord-arc curves. Sections 4--8 are the main body of the new
contribution of this paper. The aforementioned Bers coordinates by the BMO Teichm\"uller spaces
for the total space of BMO embeddings
are introduced in Section 4. The mappings $\Phi$ and $\Pi$ are also defined explicitly here by using the
Bers coordinates. From Section 5, we focus on the space ${\rm CA}$ consisting of
BMO embeddings with the chord-arc images. First, curve theoretical representations of an element of ${\rm CA}$,
arc-length parametrization and change of parameter, are discussed, and the coordinates of $L({\rm CA})$ are given
in the space of BMO functions. Theorem \ref{fundamental} and Corollary \ref{1.2} are proved in Section 6.
Properties of the biholomorphic mapping $L$ are also discussed.
Section 7 is devoted to the setup for the problem on the Riemann mapping parametrization, and
Corollary \ref{1.3} is presented there. A more important result on the real-analyticity of the inverse of
$\lambda$ is also explained. Finally in Section 8, we prove our main achievement in this paper,
Theorem \ref{1.4}. Questions about the discontinuity of certain related mappings are also answered.

\section{Preliminaries}
\subsection{BMO functions and $A_\infty$-weights}
A locally integrable complex-valued function $u$ on $\mathbb R$ is of {\it BMO} if
$$
\Vert u \Vert_*=\sup_{I \subset \mathbb R}\frac{1}{|I|} \int_I |u(x)-u_I| dx <\infty,
$$
where the supremum is taken over all bounded intervals $I$ on $\mathbb R$ and $u_I$ denotes the integral mean of $u$
over $I$. The set of all complex-valued BMO functions on $\mathbb R$ is denoted by ${\rm BMO}(\mathbb R)$.
This is regarded as a Banach space with norm $\Vert \cdot \Vert_*$
by ignoring the difference of complex constant functions.
The {\it John--Nirenberg inequality} for BMO functions (see \cite[VI.2]{Ga}) 
asserts that
there exists two universal positive constants $C_0$ and $C_{JN}$ such that for any complex-valued BMO function $u$, 
any bounded interval $I$ of $\mathbb{R}$, and any $\lambda > 0$, it holds that
\begin{equation*}\label{JN}
\frac{1}{|I|} |\{t \in I: |u(t) - u_I| \geq \lambda \}| \leq C_0 \exp\left(\frac{-C_{JN}\lambda}{\Vert u \Vert_*} \right).
\end{equation*}

A locally integrable non-negative measurable function $\omega \geq 0$ on $\mathbb R$ 
is called a weight. We say that $\omega$ is
an $A_p$-weight 
of Muckenhoupt \cite{M} for $p>1$ 
if there exists a constant $C_p(\omega) \geq 1$ such that
\begin{equation*}\label{Ap}
\left(\frac{1}{|I|} \int_I \omega(x)dx \right)\left(\frac{1}{|I|} \int_{I} \left(\frac{1}{\omega(x)}\right)^{\frac{1}{p-1}}dx\right)^{p-1}
\leq C_p(\omega)
\end{equation*}
for any bounded interval $I \subset \mathbb R$. 
We define $\omega$ to be an {\it $A_\infty$-weight} if $\omega$ is an
$A_p$-weight for some $p>1$, that is, $A_\infty=\bigcup_{p>1} A_p$.
It is known that $\omega$ is an $A_\infty$-weight if and only if
there are positive constants $\alpha(\omega)$, $K(\omega)>0$ such that  
\begin{equation}\label{SD}
\frac{\int_E \omega(x)dx}{\int_I \omega(x)dx}\leq K(\omega)\left(\frac{|E|}{|I|}\right)^{\alpha(\omega)}
\end{equation}
for any bounded interval $I \subset \mathbb{R}$ and 
for any measurable subset $E \subset I$ (see \cite[Theorem V]{CF},
\cite[Lemma VI.6.11]{Ga}). We also define $\omega$ to be an $A_1$-weight if there exists a constant $C_1(\omega) \geq 1$ such that
$$
\frac{1}{|I|} \int_I \omega(x)dx  \leq C_1(\omega)\, {\rm ess}\!\!\! \inf_{x \in I \quad}\!\!\! \omega(x)
$$
for any bounded interval $I \subset \mathbb R$. It holds $A_1 \subset A_p$ for every $p>1$.

Another characterization of $A_\infty$-weights can be given by 
the inverse Jensen inequality. Namely,
$\omega \geq 0$
belongs to the class of $A_\infty$-weights 
if and only if there exists a constant $C_\infty(\omega) \geq 1$ such that
\begin{equation*}\label{iff}
\frac{1}{|I|} \int_I \omega(x) dx \leq C_\infty(\omega) \exp \left(\frac{1}{|I|} \int_I \log \omega(x) dx \right) 
\end{equation*}
for every bounded interval $I \subset \mathbb R$ (see \cite[Theorem IV.2.15]{GR} and \cite{Hr}). 
We see that if $\omega$ is an $A_\infty$-weight on $\mathbb R$, then
$\log \omega$ belongs to ${\rm BMO}_{\mathbb R}(\mathbb R)$
which is the real subspace of ${\rm BMO}(\mathbb R)$ consisting of all
real-valued BMO functions, 
and conversely, we know the following fact
(see \cite[p.409]{GR} and \cite[Lemma VI.6.5]{Ga}).

\begin{proposition}\label{C_0}
Suppose that a weight $\omega \geq 0$ satisfies $\log \omega \in {\rm BMO}_{\mathbb R}(\mathbb R)$.
If the BMO norm $\Vert \log \omega \Vert_*$ 
is less than the constant $C_{JN}$, then $\omega$ is an $A_\infty$-weight. 
\end{proposition}

There is an example of $u \in {\rm BMO}_{\mathbb R}(\mathbb R)$ such that $e^u$ is not an $A_\infty$-weight.
Let ${\rm BMO}_{\mathbb R}^*(\mathbb R)$ denote the proper subset 
of the real subspace ${\rm BMO}_{\mathbb R}(\mathbb R)$ consisting of all real-valued BMO functions $u$
with $e^u$ being an $A_\infty$-weight. 

\begin{proposition}\label{convex}
${\rm BMO}_{\mathbb R}^*(\mathbb R)$ is a convex open subset of ${\rm BMO}_{\mathbb R}(\mathbb R)$.
\end{proposition}

\begin{proof}
To show the convexity, we use the following property of $A_\infty$-weight:
if $\omega_1$ and $\omega_2$ are $A_\infty$-weights, then $\omega_1^s\omega_2^t$ is also
an $A_\infty$-weight for $s,t \geq 0$ with $s+t=1$. This can be proved by decomposing these $A_p$-weights 
for some $p>1$ into
the product of $A_1$-weights by the Jones factorization theorem (see \cite[Corollary IV.5.3]{GR}) 
and then verifying the same claim for $A_1$-weights.
As another property of $A_\infty$-weight, we know that if $\omega$ is an $A_\infty$-weight, then
there is some $\varepsilon >0$ such that $\omega^r$ is an $A_\infty$-weight for every $r \in [0,1+\varepsilon)$
(see \cite[Theorem IV.2.7]{GR}). 
Combining these properties with the fact in Proposition \ref{C_0} that the open neighborhood of the origin of 
${\rm BMO}_{\mathbb R}(\mathbb R)$
within $C_{JN}$ is contained in ${\rm BMO}_{\mathbb R}^*(\mathbb R)$, we can 
prove that ${\rm BMO}_{\mathbb R}^*(\mathbb R)$ is open. Indeed, ${\rm BMO}_{\mathbb R}^*(\mathbb R)$ is the union of
open cones spanned by the $C_{JN}$-neighborhood of the origin
having any points of ${\rm BMO}_{\mathbb R}^*(\mathbb R)$ as their vertices. 
\end{proof}

\subsection{Strongly quasisymmetric homeomorphisms} 
A quasisymmetric homeomorphism $f:\mathbb R \to \mathbb R$ is called {\it strongly quasisymmetric} if
it is locally absolutely continuous and the derivative $f'$ is an $A_\infty$-weight.
In this case, $\log f' \in {\rm BMO}_{\mathbb R}^*(\mathbb R)$. 
Conversely, for any $u \in {\rm BMO}_{\mathbb R}^*(\mathbb R)$, the indefinite integral
$$
f_u(x) = \frac{\int_0^x e^{u(t)}dt}{\int_0^1 e^{u(t)}dt}
$$
defines a strongly quasisymmetric homeomorphism $f_u$ on $\mathbb R$ that fixes $0$, $1$ and $\infty$.
We denote the set of all
strongly quasisymmetric homeomorphisms of $\mathbb R$ onto itself with this normalization by ${\rm SQS}$.
Hence, ${\rm SQS}$ and ${\rm BMO}_{\mathbb R}^*(\mathbb R)$ correspond bijectively.
By (\ref{SD}), $f$ is a strongly quasisymmetric homeomorphism if and only if there are constants $K$ and $\alpha$ such that
\begin{equation}\label{alphaK}
\frac{|f(E)|}{|f(I)|}\leq K\left(\frac{|E|}{|I|}\right)^{\alpha}
\end{equation}
for any bounded interval $I \subset \mathbb{R}$ and 
for any measurable subset $E \subset I$. From this property, we see that ${\rm SQS}$ is a group
under the composition. 

A strongly quasisymmetric homeomorphism $f:\mathbb R \to \mathbb R$ can be also characterized by the operator on the Banach space of
BMO functions $u$. The pre-composition of $f$ to $u$ gives a change of the parameter, and we consider
this linear operator on ${\rm BMO}(\mathbb R)$ induced by $f$.

\begin{theorem}[\cite{Jo}]\label{pullback}
The increasing homeomorphism $f$ from $\mathbb R$ onto itself is strongly quasisymmetric if and only if 
the composition operator $P_f: u \mapsto u\circ f$ gives an isomorphism of ${\rm BMO}(\mathbb R)$,
that is, $P_f$ and $(P_f)^{-1}$ are bounded linear operators. 
\end{theorem}

\subsection{Quasiconformal extension and Carleson measure}
Let $M(\mathbb U)$ denote the open unit ball of the Banach space $L^{\infty}(\mathbb U)$
of all essentially bounded measurable functions on $\mathbb U$. An element in $M(\mathbb U)$ is called a
{\it Beltrami coefficient}.
We say that a measure $\lambda$ on $\mathbb U$ is a {\it Carleson measure} if
$$
\Vert \lambda \Vert_c=\sup_{I \subset \mathbb R} \frac{\lambda(I \times (0,|I|))}{|I|} <\infty,
$$
where the supremum is taken over all bounded intervals $I$ in $\mathbb R$. 
For $\mu \in L^\infty(\mathbb U)$, we set
$\lambda_\mu=|\mu(z)|^2dxdy/y$ and define
$\Vert \mu \Vert_c=\Vert \lambda_\mu \Vert_c^{1/2}$.
Then, we introduce a new norm $\Vert \mu \Vert_{\infty}+\Vert \mu \Vert_{c}$ for $\mu$.
Let $\mathcal{L}(\mathbb U)$ denote a linear subspace of $L^{\infty}(\mathbb U)$ consisting of all elements $\mu$ with 
$\Vert \mu \Vert_{\infty}+\Vert \mu \Vert_{c}<\infty$, namely, $\lambda_\mu$ is a Carleson measure on $\mathbb U$.
This is a Banach space with this norm. 
Moreover, we consider the corresponding spaces of Beltrami coefficients
as $\mathcal{M}(\mathbb U) =  M(\mathbb U) \cap \mathcal{L}(\mathbb U)$.

We consider a quasiconformal homeomorphism $F$ of $\mathbb U$ onto itself
whose complex dilatation $\mu_F=\bar \partial F/\partial F$ belongs to $\mathcal M(\mathbb U)$.
Concerning its continuous extension to the boundary $\mathbb R$, we know the following result.

\begin{theorem}[\mbox{\cite[Theorem 2.3]{FKP}}]\label{basic}
If $F$ is a quasiconformal homeomorphism of $\mathbb U$ onto $\mathbb U$ whose complex dilatation $\mu_F$
belongs to
${\mathcal M}(\mathbb U)$, then its extension to $\mathbb R$ is a strongly quasisymmetric homeomorphism.
\end{theorem}

Conversely to Theorem \ref{basic}, there is a way of extending a strongly quasisymmetric homeomorphism of $\mathbb R$
to a quasiconformal homeomorphism of $\mathbb U$ onto itself whose complex dilatation induces a Carleson measure.
Let $\phi(x)=\frac{1}{\sqrt \pi}e^{-x^2}$ and $\psi(x)=\phi'(x)=-2x \phi(x)$.
We extend a strongly quasisymmetric homeomorphism $f:\mathbb R \to \mathbb R$
to $\mathbb U$ by setting a differentiable map $F: \mathbb{U} \to \mathbb{C}$ by 
\begin{equation*}\label{F}
\begin{split}
&F(x, y) = U(x, y) + iV(x, y);\\
U(x,y)&=(f \ast \phi_y)(x),\ V(x,y)=(f \ast \psi_y)(x),
\end{split}
\end{equation*}
where $\varphi_y(x)=y^{-1} \varphi(y^{-1}x)$ for $x \in \mathbb R$ and $y>0$, and $\ast$ is the convolution.
We call this extension the variant of the {\it Beurling--Ahlfors extension}
by the heat kernel. The former statement of the next theorem follows from \cite[Theorem 4.2]{FKP}.
See \cite[Theorem 3.4]{WM-2} for its exposition. The latter statement is given in \cite{WM-4}.

\begin{theorem}\label{FKP}
For a strongly quasisymmetric homeomorphism $f:\mathbb R \to \mathbb R$, 
the map $F$ given by
the variant of the Beurling--Ahlfors extension by the heat kernel is a quasiconformal
diffeomorphism of $\mathbb U$ onto itself whose complex dilatation $\mu_F$ belongs to $\mathcal M(\mathbb U)$.
Moreover, $F$ is bi-Lipschitz with respect to the hyperbolic metric. 
\end{theorem}

We note that in the case where the BMO norm of $\log f'$ is sufficiently small for a
strongly quasisymmetric homeomorphism $f$, Semmes \cite{Se} used a modified Beurling--Ahlfors extension $F$
by compactly supported kernels $\phi$ and $\psi$ to prove the same properties as in Theorem \ref{FKP}.
By dividing the weight $f'$ into small pieces and composing the resulting maps, the assumption on the small BMO norm can be
removed to obtain a quasiconformal extension of the same properties.

\subsection{The BMO Teichm\"uller space}
The {\it universal Teichm\"uller space} $T=T(\mathbb U)$ is the set of all Teichm\"uller equivalence classes of
Beltrami coefficients in $M(\mathbb U)$. Here, $\mu_1$ and $\mu_2$ in $M(\mathbb U)$ are equivalent if
$F^{\mu_1}=F^{\mu_2}$ on $\mathbb R$, where $F^{\mu}$ denote the unique quasiconformal homeomorphism 
of $\mathbb U$ onto itself extendable to $\mathbb R$ that has complex dilatation 
$\mu \in M(\mathbb U)$ and keeps the points $0$, $1$ and $\infty$ fixed.
We denote the quotient projection by
$\pi:M(\mathbb U) \to T(\mathbb U)$, which is called the {\it Teichm\"uller projection}.
For $\mathcal M(\mathbb U) \subset M(\mathbb U)$, we define the {\it BMO Teichm\"uller space}
$T_b=T_b(\mathbb U)$ by $\pi(\mathcal M(\mathbb U))$ equipped with the quotient topology from $\mathcal M(\mathbb U)$.
This space was introduced in \cite{AZ}.
We can prove that $T_b$ has a complex Banach manifold structure 
such that $\pi:\mathcal M(\mathbb U) \to T_b(\mathbb U)$ 
is holomorphic and has a local holomorphic section (see Theorem \ref{model} below). 

Let ${\rm SQS}$ be the set of all strongly quasisymmetric homeomorphisms of $\mathbb R$ satisfying the normalized condition
keeping $0$, $1$ and $\infty$ fixed. Then, the correspondence $\mu \mapsto F^{\mu}|_{\mathbb R}$ induces 
a well-defined bijection between $T_b$ and ${\rm SQS}$.
Under this identification, structures on $T_b$ and on ${\rm SQS}$ are imported to each other.
In particular, $T_b$ is endowed with the group structure and ${\rm SQS}$ is endowed with
the complex Banach manifold structure. 

By Proposition \ref{convex}, $\rm BMO_{\mathbb R}^*(\mathbb R)$ is the open convex
subset of $\rm BMO_{\mathbb R}(\mathbb R)$ consisting of all real-valued BMO
functions $u$ on $\mathbb R$
such that $e^u$ is an $A_{\infty}$-weight. 
Concerning the bijective relation between $T_b \cong {\rm SQS}$ and $\rm BMO_{\mathbb R}^*(\mathbb R)$,
the following result was obtained in \cite[Theorem 7.1]{SWei}.

\begin{proposition}\label{topequiv}
The map $\Psi:{\rm SQS} \to \rm BMO_{\mathbb R}^*(\mathbb R)$ given by $f \mapsto \log f'$ is a surjective
homeomorphism.
\end{proposition}

This claim implies that the real BMO space provides ${\rm SQS}$, which can be regarded as
the real model of the BMO Teichm\"uller space $T_b$, 
with a real Banach manifold structure. Later in Corollary \ref{real-analytic}, we will see 
that this map is a real-analytic homeomorphism. 
In particular, the topology on $T_b \cong {\rm SQS}$ is equivalent to that on 
${\rm BMO}_{\mathbb R}^*(\mathbb R)$ defined by the BMO norm. 
Thus, we have the identifications
\begin{equation}\label{trinity}
T_b \cong {\rm SQS} \cong {\rm BMO}_{\mathbb R}^*(\mathbb R),
\end{equation}
and their correspondence becomes clear.

Especially, we consider the group structure on $T_b$. As we have mentioned,
we can regard $T_b$ as a group by the identification of $T_b$ with ${\rm SQS}$.
The group operation is denoted by $[\mu] \ast [\nu]$ for $[\mu], [\nu] \in T_b$ and the inverse by
$[\mu]^{-1}$. More explicitly, $[\mu] \ast [\nu]$ is the Teichm\"uller class of
the complex dilatation $\mu \ast \nu$ of $F^\mu \circ F^\nu$ and $[\mu]^{-1}$ is that of
the complex dilatation $\mu^{-1}$ of $(F^\mu)^{-1}$.
The following result, which says that $T_b$ is a {\it partial topological group} in the sense of \cite{GS}, 
was proved in \cite{Wei}.

\begin{proposition}\label{nearid}
If $[\mu]$ and $[\nu]$ converge to $[\rm 0]$ in $T_b$, 
then $[\mu] \ast [\nu] \to [\rm 0]$ and $[\nu]^{-1} \to [\rm 0]$.
\end{proposition}

The correspondence $f \mapsto \log f'$ for $f \in {\rm SQS}$ gives
a topological equivalence of 
the BMO Teichm\"uller space $T_b$ with ${\rm BMO}_{\mathbb R}^*(\mathbb R)$ by Proposition \ref{topequiv}.
Translating Proposition \ref{nearid} to ${\rm BMO}_{\mathbb R}^*(\mathbb R)$,
we see that $\Vert \log (h\circ f)' \Vert_* \to 0$ as $\Vert \log h' \Vert_* \to 0$ and
$\Vert \log f' \Vert_* \to 0$ for $h, f \in {\rm SQS}$. 

\subsection{Conformal extension and BMOA}
For $\mu \in M(\mathbb L)$, we take a quasiconformal homeomorphism $H=H_\mu$ of
$\mathbb C$ onto itself that is conformal on $\mathbb U$ and whose complex dilatation on $\mathbb L$
is $\mu$. We can characterize the condition $\mu \in \mathcal M(\mathbb L)$ 
by the conformal map $H|_{\mathbb U}$ and analytic function spaces defined as follows.
Let $H^2(\mathbb U)$ be the Hardy space of holomorphic functions on $\mathbb U$ defined in a conformally invariant way
from that on $\mathbb D$ (see \cite[Chapter 11]{Du}).
We introduce 
$$
{\rm BMOA}(\mathbb U)=\{\phi \in H^2(\mathbb U) \mid \Vert \lambda^1_\phi \Vert_c<\infty\},
$$
where $\lambda^1_\phi=|\phi'(z)|^2ydxdy$ is a Carleson measure on $\mathbb U$. 
This space modulo constants is a Banach space with
norm $\Vert \lambda^1_\phi \Vert_c^{1/2}$. 
For the set ${\rm Hol}(\mathbb U)$ of holomorphic functions on $\mathbb U$, we also define
$$
B(\mathbb U)=\{\varphi \in {\rm Hol}(\mathbb U) \mid \Vert \lambda^2_\varphi \Vert_c<\infty\},
$$
where $\lambda^2_\varphi=|\varphi(z)|^2y^3dxdy$ is a Carleson measure on $\mathbb U$. This is a Banach space with
norm $\Vert \lambda^2_\varphi \Vert_c^{1/2}$. 

With the aid of Theorems \ref{basic} and \ref{FKP} and the geometric characterization of the image $H(\mathbb U)$, 
we can summarize the equivalent conditions for $H=H_\mu$
as follows (see \cite[Theorem A]{SWei}).

\begin{theorem}[\mbox{\cite[Theorem 4]{AZ}, \cite[Theorem 4]{BJ}}]\label{Guo11}
Let $H:\mathbb U \to \mathbb C$ be a conformal mapping on $\mathbb U$ with
$\lim_{z \to \infty}H(z)=\infty$ that extends to a quasiconformal homeomorphism of 
$\mathbb C$ whose complex dilatation on $\mathbb L$ is $\mu$.
Then, the following conditions are equivalent $\!:$ 
\begin{enumerate}
\item[(a)]
$\mu$ belongs to $\mathcal M(\mathbb L);$
\item[(b)]
$\mathcal L_H =\log H'\in {\rm BMOA}(\mathbb U);$
\item[(c)]
$\mathcal S_H = \mathcal L_H'' - \frac{1}{2}(\mathcal L_H')^2 \in B(\mathbb U)$;
\item[(d)]
$\Gamma = \partial H(\mathbb U)$ is a quasicircle satisfying the Bishop--Jones condition. 
\end{enumerate}
\end{theorem}

Here, a Jordan curve $\Gamma$ satisfies the {\it Bishop--Jones condition} if 
there are constants $\delta>0$ and $C>0$ such that
the domain $\Omega$ bounded by $\Gamma$ 
satisfies the following: 
for every $z \in \Omega$, 
there exists a Jordan domain $\Omega_z \subset \Omega$ containing $z$ and having 
the rectifiable boundary $\partial \Omega_z$ such that the length of $\partial \Omega_z$ is less than $C d(z,\Gamma)$ and
the harmonic measure of $\partial \Omega_z \cap \Gamma$ with respect to $z \in \Omega_z$ is
greater than $\delta$. This condition is invariant under a
bi-Lipschitz homeomorphism of $\mathbb C$ onto itself in the Euclidean metric.

We consider a subset ${\mathcal T}_b \subset {\rm BMOA}(\mathbb U)$ consisting of all $\mathcal L_H$ such that
$H:\mathbb U \to \mathbb C$ is a conformal mapping on $\mathbb U$ with
$\lim_{z \to \infty}H(z)=\infty$ that can be extended quasiconformally to $\mathbb C$.
We also consider a subset ${\mathscr T}_b \subset B(\mathbb U)$ consisting of all $\mathcal S_H$
for those $H$. It is known that ${\mathcal T}_b$ and ${\mathscr T}_b$ are contractible open subsets
that correspond bijectively to $T_b$. Hence, they serve as the models of the BMO Teichm\"uller space.
The correspondence of the elements in these spaces is described as follows, and in particular,
the complex Banach manifold structure on $T_b$ can be introduced in this way.
The following theorem is based on the arguments in \cite[Sections 5, 6]{SWei} 
with the adaptation to the case of $\mathbb U$ as in \cite[Section 7]{WM-4}.

\begin{theorem}\label{model}
$(1)$ The map $\beta:{\mathcal M}(\mathbb L) \to {\mathscr T}_b$ defined by $\mu \mapsto \mathcal S_{H_\mu}$
is holomorphic and there is a local holomorphic inverse of $\beta$ at every point of ${\mathscr T}_b$.
Moreover, $\beta \circ \pi^{-1}$ for the Teichm\"uller projection $\pi:{\mathcal M}(\mathbb L) \to T_b$ gives a
well-defined surjective homeomorphism $T_b \to {\mathscr T}_b$.
$(2)$ The map $\alpha:{\mathcal T}_b \to {\mathscr T}_b$ defined by $\phi \mapsto \varphi=\phi''-\frac{1}{2} (\phi')^2$ is 
a biholomorphic homeomorphism.
\end{theorem}

Concerning the boundary extension of $\phi \in {\rm BMOA}(\mathbb U) \subset H^2(\mathbb U)$, 
we note that $\phi$ has the non-tangential limit 
almost everywhere on $\mathbb R$ and the Poisson integral of this boundary function reproduces $\phi$.
This links BMO properties of $\phi$ on $\mathbb U$ and on $\mathbb R$.
The following theorem is well-known, which can be seen from \cite[Theorems 9.17 and 9.19]{Zh}.

\begin{theorem}\label{131} 
Let $b(\phi)$ be the boundary extension of $\phi \in {\rm BMOA}(\mathbb U)$ defined by
the non-tangential limit on $\mathbb R$.
Then, $b(\phi) \in {\rm BMO}(\mathbb R)$, and
the boundary extension operator $b:{\rm BMOA}(\mathbb U) \to  {\rm BMO}(\mathbb R)$ is 
an isomorphism onto the image.
\end{theorem}

\section{BMO embeddings and chord-arc curves}\label{curves}
We generalize strongly quasisymmetric homeomorphisms $\mathbb R \to \mathbb R$ to 
BMO embeddings $\gamma:\mathbb R \to \mathbb C$ and consider those whose images are chord-arc curves.
These are defined below, but we note here that a BMO embedding is a mapping $\gamma$ of $\mathbb R$, that is, not only its image 
$\Gamma=\gamma(\mathbb R)$ but also with its parametrization, whereas a chord-arc curve refers to the image $\Gamma$ 
of a certain special embedding $\gamma$. 

\begin{definition}
A homeomorphic embedding $\gamma:\mathbb R \to \mathbb C$ passing through $\infty$ is called a
{\it BMO embedding} if there is a quasiconformal homeomorphism $G$ of $\mathbb C$ onto itself with $G|_{\mathbb R}=\gamma$
whose complex dilatation $\mu=\bar \partial G/\partial G$ satisfies
$\mu|_{\mathbb U} \in \mathcal M(\mathbb U)$ and $\mu|_{\mathbb L} \in \mathcal M(\mathbb L)$.
Such a map $G$ is called a {\it BMO quasiconformal homeomorphism}.
\end{definition}

\begin{definition}
The image $\Gamma=\gamma(\mathbb R)$ of a homeomorphic embedding $\gamma:\mathbb R \to \mathbb C$ 
passing through $\infty$ is called a {\it chord-arc curve} if $\Gamma$ is locally rectifiable and there exists
a constant $K \geq 1$ such that the length of the arc $\gamma([a,b])$
for any $a, b \in \mathbb R$ with $a<b$ is bounded by $K|\gamma(a)-\gamma(b)|$.
\end{definition}

In a similar way, we can also define bounded BMO embeddings and bounded chord-arc curves, 
which are mappings of the unit circle $\mathbb S$ into $\mathbb C$ not passing through $\infty$.
However, for the convenience of the arguments, we consider the unbounded case in this paper.

The image $\Gamma=\gamma(\mathbb R)$ of a homeomorphic embedding $\gamma:\mathbb R \to \mathbb C$ 
passing through $\infty$ is called a 
quasicircle if $\Gamma$ is the image of $\mathbb R$ under a quasiconformal homeomorphism of $\mathbb C$.
This is known to be equivalent to
satisfying a weaker condition than the above
by replacing the length of $\gamma([a,b])$ with the diameter of $\gamma([a,b])$ 
even though $\Gamma$ is not necessarily locally rectifiable (see \cite[Theorems IV.4, 5]{Ah}). 
Hence, a chord-arc curve is a quasicircle. The corresponding characterization of a chord-arc by
the image of $\mathbb R$ was shown in \cite[Proposition 1.13]{JK} as follows.

\begin{proposition}\label{biLip}
$\Gamma$ is chord-arc curve if and only if $\Gamma$ is 
the image of $\mathbb R$ under a bi-Lipschitz homeomorphism of $\mathbb C$ with respect to the Euclidean metric.
\end{proposition}

We prepare the following lemma used throughout this paper. 
The statement is known to be true and has been already used before in several places.
This result originates in \cite[Lemma 10]{CZ}. There is an exposition in \cite{WM-5}.

\begin{lemma}\label{composition}
Let $F$ be a quasiconformal homeomorphism of $\mathbb U$ onto itself
that is bi-Lipschitz
with respect to the hyperbolic metric whose complex dilatation $\nu$ is in $\mathcal M(\mathbb U)$.
Then, the complex dilatation of
$F^{-1}$ belongs to $\mathcal M(\mathbb U)$.
Let $H$ be a quasiconformal homeo\-morphism of $\mathbb U$ into $\mathbb C$ 
whose complex dilatation $\mu$ is in $\mathcal M(\mathbb U)$. Then, the complex dilatation of
$H \circ F$, denoted by $F^*\mu$, also belongs to $\mathcal M(\mathbb U)$.
\end{lemma}

From this lemma, we also see that if $f:\mathbb R \to \mathbb R$ is 
a strongly quasisymmetric homeomorphism and $h:\mathbb R \to \mathbb C$
is a BMO embedding, then $h \circ f:\mathbb R \to \mathbb C$ is also a BMO embedding.
Indeed, by Theorem \ref{FKP}, we can take the bi-Lipschitz quasiconformal extension $F$ of $f$.

We first consider the derivative of a BMO embedding $\gamma:\mathbb R \to \mathbb C$.
The following claim has been proved in a more general setting in \cite[Theorem 6.2]{Mac}, but to present 
a typical argument in this section and also introduce the notation, we show
our proof here.

\begin{proposition}\label{curve}
A BMO embedding 
$\gamma:\mathbb R \to \mathbb C$ has its derivative $\gamma'$ almost everywhere on $\mathbb R$
and $\log \gamma'$
belongs to ${\rm BMO}(\mathbb R)$. 
\end{proposition}

\begin{proof}
Let $G:\mathbb C \to \mathbb C$ be a BMO quasiconformal homeomorphism associated with $\gamma$,
and set
$\mu_1=\mu|_{\mathbb U} \in \mathcal M(\mathbb U)$ and $\mu_2=\mu|_{\mathbb L} \in \mathcal M(\mathbb L)$
for the complex dilatation $\mu$ of $G$.
We take a quasiconformal homeomorphism $F:\mathbb C \to \mathbb C$ whose complex dilatation is
$\mu_1(z)$ for $z \in \mathbb U$ and $\overline{\mu_1(\bar z)}$ for $z \in \mathbb L$, which maps $\mathbb R$ onto itself.
By Theorem \ref{basic}, $f=F|_{\mathbb R}$ is strongly quasisymmetric and $\log f'$ belongs to ${\rm BMO}_{\mathbb R}^*(\mathbb R)$.
By Theorem \ref{FKP}, we may assume that $F$ is bi-Lipschitz on $\mathbb L$ by replacing the quasiconformal extension of $f$.
Next, we take a quasiconformal homeomorphism $H:\mathbb C \to \mathbb C$ that is conformal on $\mathbb U$ and
whose complex dilatation on $\mathbb L$ is the push-forward $F_* \mu_2$ of $\mu_2$ by $F$. Namely, 
the complex dilatation of $H \circ F|_{\mathbb L}$ is $\mu_2$. Then, $H \circ F$ coincides with $G$ up to
an affine transformation of $\mathbb C$, and hence, we may assume that $H \circ F=G$.

The complex dilatation $F_* \mu_2=(F^{-1})^* \mu_2$ belongs to $\mathcal M(\mathbb L)$ by Lemma \ref{composition}.
Then, 
${\mathcal L}_{H|_{\mathbb U}}=\log (H|_{\mathbb U})' \in {\rm BMOA}(\mathbb U)$ by Theorem \ref{Guo11}, and
it follows from Theorem \ref{131} that the boundary function $b(\log (H|_{\mathbb U})')$ 
defined by the non-tangential limit of $\log (H|_{\mathbb U})'$ belongs to
${\rm BMO}(\mathbb R)$. 
Since $\log (H|_{\mathbb U})'$ has the finite non-tangential limit 
almost everywhere on $\mathbb R$, so does $(H|_{\mathbb U})'$.
This implies that $H|_{\mathbb U}$ has a finite angular derivative almost everywhere
on $\mathbb R$ by \cite[Proposition 4.7]{Pom}. However, since $H(\mathbb U)$ is a quasidisk,
\cite[Theorem 5.5]{Pom} asserts that the angular derivative at $x \in \mathbb R$ coincides with
$$
h'(x)=\lim_{\mathbb R \ni \xi \to x} \frac{h(\xi)-h(x)}{\xi-x}
$$
for $h=H|_{\mathbb R}$.
This shows that the non-tangential limit of $(H|_{\mathbb U})'$ coincides with the ordinary derivative $h'$
almost everywhere on $\mathbb R$. By taking the logarithm, we have
$b(\log (H|_{\mathbb U})')=\log h'$.

By $H \circ F=G$, we see that $\gamma=G|_{\mathbb R}$ has the derivative $\gamma'$ almost everywhere on $\mathbb R$, 
and satisfies
$$
\log h' \circ f + \log f'=\log \gamma'.
$$
We have seen that $\log f' \in {\rm BMO}(\mathbb R)$.
Since $f$ is strongly quasisymmetric and $\log h' \in {\rm BMO}(\mathbb R)$,
Theorem \ref{pullback} shows that $\log h' \circ f \in {\rm BMO}(\mathbb R)$.
Thus, we
obtain that $\log \gamma' \in {\rm BMO}(\mathbb R)$.
\end{proof}

Theorem \ref{Guo11} implies that any quasicircle $\Gamma$ satisfying the Bishop--Jones condition is 
the image $\gamma(\mathbb R)$ of some BMO embedding $\gamma$. 
(Conversely, 
the image $\gamma(\mathbb R)$ of any BMO embedding $\gamma$ 
is a quasicircle $\Gamma$ satisfying the Bishop--Jones condition. 
Indeed, in the proof of Proposition \ref{curve}, we see that $\Gamma=\gamma(\mathbb R)=h(\mathbb R)$ for 
$h=H|_{\mathbb R}$ with $\log (H|_{\mathbb U})' \in {\rm BMOA}(\mathbb U)$. Then by Theorem \ref{Guo11},
$\Gamma$ is a quasicircle satisfying the Bishop--Jones condition.) Since
a chord-arc curve is the image of $\mathbb R$
under some bi-Lipschitz homeomorphism of $\mathbb C$ by Proposition \ref{biLip} and
since the Bishop--Jones condition is invariant under a bi-Lipschitz homeomorphism of $\mathbb C$, 
any chord-arc curve satisfies the Bishop--Jones condition. Thus, we have:

\begin{proposition}\label{image}
Any chord-arc curve $\Gamma$ is the image $\gamma(\mathbb R)$ of some BMO embedding $\gamma$.
\end{proposition}

The converse is not true. In fact, there are examples of 
BMO embeddings whose images are not locally rectifiable (see \cite{Bi} and \cite{Se0}). 
On the contrary, we can show its converse if we assume the $A_{\infty}$ property on the derivative. 
The following claim in 
\cite[Theorem 4.2]{JK} has an essential role for this argument. 
In the case of the unit disk,
we may also refer to \cite[Theorem 7.11]{Pom}.
Notations are the same as those in the proof of Proposition \ref{curve}, and conditions (a) and (b) below are
always satisfied in our setting of BMO embeddings.

\begin{lemma}\label{JK4.2}
Suppose that $\Gamma$ is a locally rectifiable Jordan curve passing through $\infty$
and $H$ is a conformal homeomorphism of $\mathbb U$ onto the left domain bounded by $\Gamma$
with its extension $h$ to $\mathbb R$. If 
\begin{enumerate}
\item[(a)]
$\Gamma$ is a quasicircle,
\item[(b)]
$\log |H'(z)|$ is represented by the Poisson integral of its boundary function, and
\item[(c)]
$|h'|$ is an $A_{\infty}$-weight,
\end{enumerate}
then $\Gamma$ is a chord-arc curve. 
\end{lemma}

We obtain the characterization of a chord-arc curve as follows.

\begin{theorem}\label{Ainfty}
The image of a BMO embedding $\gamma:\mathbb R \to \mathbb C$ is a chord-arc curve if and only if 
$|\gamma'|$ is an $A_\infty$-weight on $\mathbb R$. Moreover, $\gamma$ is locally absolutely continuous in this case.
\end{theorem}

\begin{proof}
We still use the same notations as in
the proof of Proposition \ref{curve}. We consider the quasidisk $H(\mathbb U)$, and
suppose that this is bounded by the chord-arc curve $\Gamma=\gamma(\mathbb R)$. In this case,
the angular derivative of $H|_{\mathbb U}$ coincides with the ordinary derivative $h'$ on $\mathbb R$.
Since the arc-length of $\Gamma$ is given by the integration of $|h'|$,  
the theorem of Lavrentiev (see \cite[p.222]{JK}) implies that
$|h'|$ is an $A_\infty$-weight. 
By $H \circ F=G$, we have
$|h'| \circ f \cdot f'=|\gamma'|$ on $\mathbb R$.
Let $\eta(x)=\int_0^x |h'(t)|dt$, which is a strongly quasisymmetric homeomorphism of $\mathbb R$.
Then, so is the composition $\eta \circ f$ whose derivative coincides with $|\gamma'|$. 
Thus, we see that $|\gamma'|$ is an $A_\infty$-weight.

We borrow Lemma \ref{JK4.2} to prove the sufficiency.  Suppose that $|\gamma'|$ is an $A_\infty$-weight, which is equivalent to that $|h'|$ is an $A_\infty$-weight.
By ${\mathcal L}_{H|_{\mathbb U}}=\log (H|_{\mathbb U})' \in {\rm BMOA}(\mathbb U)$, we see that $\log |(H|_{\mathbb U})'|$ is the Poisson integral of its non-tangential limit. Thus, it remains to show the Jordan curve $\Gamma$ is locally rectifiable. We point out that in the case that $\Gamma$ is the image of the unit circle $\mathbb S$ with $\infty \notin \Gamma$, this is simple. Indeed, $\log (H|_{\mathbb D})' \in {\rm BMOA}(\mathbb D)$ implies 
that $(H|_{\mathbb D})'$ is an outer function. Then, $(H|_{\mathbb D})' \in H^1(\mathbb D)$ if and only if $|h' | \in L^1(\mathbb S)$
(see \cite[p.64]{Ga}). Since $|h'| \in L^{1+\varepsilon}(\mathbb S)$ for some $\varepsilon >0$
by a property of $A_\infty$-weight, we see that
$(H|_{\mathbb D})'$ belongs to the Hardy space $H^1(\mathbb D)$, and then $\Gamma$ is rectifiable
(see \cite[Theorem 6.8]{Pom}). However, in our case, we need more arguments which are given in the following. 

We define a strongly quasisymmetric homeomorphism $\eta(x)=\int_0^x |h'(t)|dt$. 
Let $T:\overline{\mathbb U} \to \overline{\mathbb D}$ be the Cayley transformation defined by
$T(z)=(z-i)/(z+i)$.
Then, the conjugate $\tilde \eta=T \circ \eta \circ T^{-1}$ is a strongly quasisymmetric homeomorphism of $\mathbb S$ 
fixing $1$ and $-1$, and $|\tilde \eta'|$ is an $A_\infty$-weight on $\mathbb S$.
By computation, we have
$$
|\tilde \eta'|=\frac{|T'| \circ \eta \circ T^{-1}}{|T'| \circ T^{-1}} \cdot \eta' \circ T^{-1}
= \frac{|T'| \circ (T^{-1} \circ \tilde \eta)}{|T'| \circ T^{-1}} \cdot |h'| \circ T^{-1}.
$$
Hence, 
\begin{equation}\label{Holder}
\begin{split}
|h'| \circ T^{-1}(\xi) &=|\tilde \eta'(\xi)|\cdot \frac{|T'| \circ T^{-1}(\xi)}{|T'| \circ (T^{-1} \circ \tilde \eta(\xi))}\\
& =|\tilde \eta'(\xi)| \cdot \frac{|(T^{-1})'| \circ \tilde \eta(\xi)}{|(T^{-1})'|(\xi)}\\
&=|\tilde \eta'(\xi)| \cdot \frac{|\xi - 1|^2}{|\tilde \eta(\xi) - \tilde\eta(1)|^2}
\lesssim |\tilde \eta'(\xi)||\xi-1|^{2(1-\frac{1}{\alpha})}.
\end{split}
\end{equation}
Here, in the last step, we used the H\"older continuity of 
the quasisymmetric homeomorphism $\tilde \eta^{-1}$ for some $\alpha \in (0,1)$ (see \cite[Theorem III.2]{Ah}).  
Set $m = 2(\frac{1}{\alpha} - 1)>0$. 

By the conformal invariance of $\rm BMOA$,
$\log (H|_{\mathbb U})'\circ T^{-1}$ belongs to ${\rm BMOA}(\mathbb D)$. 
It is well-known that $\log(w-1)$ belongs to ${\rm BMOA}(\mathbb D)$ (see \cite{Dan}).
Hence, these can be represented by the Poisson integral. 
Then, we have
\begin{equation*}
\begin{split}
|H'\circ T^{-1}(w)(w-1)^m| &=|\exp[\,\log H'\circ T^{-1}(w)+m\log(w-1)]|\\
&=\left| \exp \left[\frac{1}{2\pi} \int_{\mathbb S}(\log h' \circ T^{-1}(\xi)+m\log(\xi-1))\,{\rm Re}\,\left(\frac{\xi+w}{\xi-w}\right)|d\xi|\right]\right|\\
&=\exp \left[\frac{1}{2\pi} \int_{\mathbb S}\log (|h'| \circ T^{-1}(\xi)|\xi-1|^m)\,{\rm Re}\,\left(\frac{\xi+w}{\xi-w}\right)|d\xi|\right]\\
&\leq\frac{1}{2\pi}\int_{\mathbb S} |h'| \circ T^{-1}(\xi)|\xi-1|^m \,{\rm Re}\,\left(\frac{\xi+w}{\xi-w}\right)|d\xi|,
\end{split}
\end{equation*}
where the last inequality is due to the Jensen inequality.
Then, by the estimate in \eqref{Holder} and the fact 
that the $A_\infty$-weight $|\tilde \eta'|$ is in $L^{1}(\mathbb S)$, 
the last Poisson integral above is convergent. From this, we see that 
$H' \circ T^{-1}(w)(w-1)^m$ belongs to $H^1(\mathbb D)$.
 
Let $I$ be a closed interval in $\mathbb S\setminus \{1\}$ and let $0<\theta_0 < \theta_1 < \cdots <\theta_n<2\pi$
be any finite sequence such that $e^{i\theta_{\nu}} \in I$ for $\nu=0,1,\ldots,n$ with
$e^{i\theta_0}$ and $e^{i\theta_n}$ being the endpoints of $I$. For $M_I=\max_{\xi \in I}2|\xi-1 |^{-2-m}$,
we have 
\begin{equation*}
\begin{split}
\sum_{\nu = 1}^{n} |h\circ T^{-1}(e^{i\theta_{\nu}}) - h\circ T^{-1}(e^{i\theta_{\nu - 1}})| 
& = \lim_{r \to 1^-}\sum_{\nu = 1}^{n} |H \circ T^{-1}(re^{i\theta_{\nu}}) - H\circ T^{-1}(re^{i\theta_{\nu - 1}})|\\
& \leq \lim_{r \to 1^-} \int_{I}|(H\circ T^{-1})'(r\xi)| |d\xi|\\
&\leq M_I \lim_{r \to 1^-} \int_{I} |H'\circ T^{-1}(r\xi)(r\xi-1)^m| |d\xi|\\
& = M_I \int_{I} |h'\circ T^{-1}(\xi)(1 - \xi)^m| |d\xi|.
\end{split}
\end{equation*}
This integral is bounded by the norm of $H' \circ T^{-1}(w)(w-1)^m \in H^1(\mathbb D)$.
Then by the definition of arc length, we see that
$h \circ T^{-1}(\mathbb S \setminus \{1\})=h(\mathbb R)=\Gamma$ is locally rectifiable. More strongly,
it is easy to see that
this estimate implies that $h\circ T^{-1}$ is locally absolutely continuous on $\mathbb S \setminus \{1\}$.
See \cite[Theorem 10.11]{Pom75}.
Then, $h$ is locally absolutely continuous on $\mathbb R$, and since 
$\gamma=h \circ f$ and $f$ is locally absolutely continuous, so is $\gamma$.
\end{proof}

\begin{corollary}
A Jordan curve $\Gamma$ passing through $\infty$ is a chord-arc curve if and only if $\Gamma$ 
is the image of some BMO embedding $\gamma:\mathbb R \to \mathbb C$ such that
$|\gamma'|$ is an $A_{\infty}$-weight.
\end{corollary}

\begin{remark}
It was shown in \cite[p.877]{Mac} that being a chord-arc curve is a M\"obius invariant. In particular,
$\Gamma$ is a chord-arc curve if and only if $T(\Gamma)$ is a bounded chord-arc curve.
Then, by Theorem \ref{Ainfty}, we see that $|\gamma'|$ is an $A_\infty$-weight on $\mathbb R$
if and only if $|(T \circ \gamma \circ T^{-1})'|$ is an $A_\infty$-weight on $\mathbb S$.
\end{remark}

Furthermore,
if we assign a proper parametrization to a chord-arc curve,
then this mapping itself is realized as a BMO embedding. This result,
which is stated precisely below, was proved in \cite[Theorem 0.2]{Se} where such a mapping was called
a {\it strongly quasisymmetric embedding}.
This can be derived also from Theorem \ref{Ainfty} 
if we assume Proposition \ref{image}, which is a consequence from several crucial results as we have seen.

\begin{theorem}\label{0.2}
Suppose $\gamma$ maps $\mathbb R$ homeomorphically onto a chord-arc curve $\Gamma$. If 
$\gamma$ is locally absolutely continuous and $|\gamma'|$ is an $A_{\infty}$-weight on $\mathbb R$, then
$\gamma$ is a BMO embedding. 
\end{theorem}

\begin{proof}
By Proposition \ref{image}, there is a BMO embedding $h:\mathbb R \to \Gamma$,
and by Theorem \ref{Ainfty},
$h$ is absolutely continuous and $|h'|$ is an $A_\infty$-weight. We set $f=h^{-1} \circ \gamma$,
which is an increasing homeomorphism of $\mathbb R$ onto itself. 
From $h \circ f=\gamma$, we see that $f$ maps a set of null measure to a set of null measure
because $\gamma$ is locally absolutely continuous and $|h'(x)| > 0$ almost everywhere on $\mathbb R$.
Hence, $f$ is locally absolutely continuous.
Taking the derivative, we have $|h'| \circ f \cdot f'=|\gamma'|$.
Then, for strongly quasisymmetric homeomorphisms
$$
\tilde h(x)=\int_0^x |h'(t)|dt \quad{\rm and} \quad \tilde \gamma(x)=\int_0^x |\gamma'(t)|dt
$$
of $\mathbb R$, 
we have $\tilde h \circ f=\tilde \gamma$. This implies that $f$ is also a strongly quasisymmetric homeomorphism of $\mathbb R$.

By Theorem \ref{FKP}, we can choose a quasiconformal extension $F:\mathbb C \to \mathbb C$ of $f$ such that 
its complex dilatation $\mu$ satisfies 
$\mu|_{\mathbb U} \in \mathcal M(\mathbb U)$, $\mu|_{\mathbb L} \in \mathcal M(\mathbb L)$ and that
$F|_{\mathbb U}$ and $F|_{\mathbb L}$ are bi-Lipschitz homeomorphisms in the hyperbolic metric on $\mathbb U$ and $\mathbb L$
respectively. We also choose a BMO quasiconformal homeomorphism $H$ of $\mathbb C$ with $H|_{\mathbb R}=h$.
Then by Lemma \ref{composition}, $H \circ F$ is a BMO quasiconformal homeomorphism whose restriction to $\mathbb R$ is $h \circ f=\gamma$.
Hence, $\gamma$ is a BMO embedding.
\end{proof}

Recall that ${\rm BMO}_{\mathbb R}^*(\mathbb R)$ denotes the convex open subset 
of the real subspace ${\rm BMO}_{\mathbb R}(\mathbb R)$ consisting of real-valued BMO functions $u$
with $e^u$ being an $A_\infty$-weight. 
The above argument implies that
$\log |\gamma'|={\rm Re}\, (\log \gamma')$ belongs to ${\rm BMO}_{\mathbb R}^*(\mathbb R)$ for any BMO embedding
$\gamma:\mathbb R \to \mathbb C$ whose image is
a chord-arc curve. Conversely, if a BMO embedding $\gamma$ in particular satisfies $|\gamma'|=1$, which means that
$\gamma$ is parametrized by its arc-length, then $\gamma(\mathbb R)$ is a chord-arc curve.

\section{The Bers coordinates using Teichm\"uller spaces}

We provide canonical coordinates for the space of BMO embeddings and regard this space
as a complex Banach manifold. This is done by a standard method due to Bers in Teich\-m\"uller theory
for investigating quasiconformal deformation spaces such as quasifuchsian spaces.
We point out that this coordinate system has never been used until we recently introduced, though it is very natural and
is more powerful for the study of curves than expected. To make those arguments applicable, we deal with 
the family of all BMO embeddings, which are not just curves but also with their parametrizations.

We impose the normalization $\gamma(0)=0$, $\gamma(1)=1$, and $\gamma(\infty)=\infty$ on every
BMO embedding $\gamma$. Let ${\rm BE}$ be the set of all normalized BMO embeddings.
For $\mu_1 \in \mathcal M(\mathbb U)$ and $\mu_2 \in \mathcal M(\mathbb L)$,
we denote by $G(\mu_1,\mu_2)$ the normalized BMO quasiconformal homeomorphism $G$ of $\mathbb C$ 
($G(0)=0$, $G(1)=1$, and $G(\infty)=\infty$) whose complex dilatation $\mu$
satisfies $\mu|_{\mathbb U}=\mu_1$ and $\mu|_{\mathbb L}=\mu_2$.
Then, we define a map
$$
\widetilde \iota:\mathcal M(\mathbb U) \times \mathcal M(\mathbb L) \to {\rm BE}
$$
by $\widetilde \iota(\mu_1,\mu_2)=G(\mu_1,\mu_2)|_{\mathbb R}$.

The BMO Teichm\"uller space $T_b(\mathbb U)$ on the upper half-plane $\mathbb U$ is
defined to be
the set of all Teichm\"uller equivalence classes $[\mu]$ for $\mu \in \mathcal M(\mathbb U)$.
The Teichm\"uller projection $\pi:\mathcal M(\mathbb U) \to T_b(\mathbb U)$ is defined by the
quotient map $\mu \mapsto [\mu]$.
The BMO Teichm\"uller space $T_b(\mathbb L)$ on the lower half-plane $\mathbb L$ 
and related concepts are
defined similarly. Then, by the argument of simultaneous uniformization due to Bers,
we see the following fact.

\begin{proposition}\label{id}
The space ${\rm BE}$ of normalized BMO embeddings is identified with
$T_b(\mathbb U) \times T_b(\mathbb L)$. More precisely, $\widetilde \iota$ splits into a well-defined bijection
$$
\iota:T_b(\mathbb U) \times T_b(\mathbb L) \to {\rm BE}
$$
by the product of the Teichm\"uller projections
$$
\widetilde \pi:\mathcal M(\mathbb U) \times \mathcal M(\mathbb L) \to T_b(\mathbb U) \times T_b(\mathbb L),
$$
such that $\widetilde \iota=\iota \circ \widetilde \pi$.
\end{proposition}

We call the representation of $\rm BE$ by $T_b(\mathbb U) \times T_b(\mathbb L)$ the {\it Bers coordinates} 
of ${\rm BE}$. Moreover, we provide complex Banach manifold structures for $T_b(\mathbb U)$ and $T_b(\mathbb L)$ by
using the pre-Schwarzian derivative models as in Theorem \ref{model}. Namely, $T_b(\mathbb U)$ is identified with the domain
$\mathcal T_b(\mathbb L)$ 
of ${\rm BMOA}(\mathbb L)$, and $T_b(\mathbb L)$ is identified with the domain $\mathcal T_b(\mathbb U)$
of ${\rm BMOA}(\mathbb U)$:
\begin{align*}
T_b(\mathbb U) &\cong \mathcal T_b(\mathbb L)=\{ {\mathcal L}_{G(\mu,0)} \in {\rm BMOA}(\mathbb L) \mid \mu \in \mathcal M(\mathbb U)\};\\
T_b(\mathbb L) &\cong \mathcal T_b(\mathbb U)=\{ {\mathcal L}_{G(0,\mu)} \in {\rm BMOA}(\mathbb U) \mid \mu \in \mathcal M(\mathbb L)\}.
\end{align*}
Then, by the identification ${\rm BE} \cong T_b(\mathbb U) \times T_b(\mathbb L)$ in
Proposition \ref{id}, we may also regard ${\rm BE}$ as a domain 
of ${\rm BMOA}(\mathbb L) \times {\rm BMOA}(\mathbb U)$.

Here, we review the canonical biholomorphic automorphisms of the BMO Teichm\"uller space
$T_b(\mathbb U)$. Let $F^\nu$ be the normalized quasiconformal homeomorphism of $\mathbb U$ onto itself
whose complex dilatation is $\nu \in \mathcal M(\mathbb U)$. By Theorem \ref{FKP},
we can make $F^\nu$ bi-Lipschitz in the hyperbolic metric 
by choosing some $\nu$ in any Teichm\"uller class $[\nu] \in T_b(\mathbb U)$. 
We define the right translation $r_\nu$ of $\mathcal M(\mathbb U)$ by 
$r_\nu(\mu)=\mu \ast \nu$ for every $\mu \in \mathcal M(\mathbb U)$, where
$\mu \ast \nu$ denotes the complex dilatation of $F^{\mu} \circ F^{\nu}$.
By Lemma \ref{composition}, $r_\nu(\mu) \in \mathcal M(\mathbb U)$, and since the inverse of
$r_\nu$ is $r_{\nu^{-1}}$ where $\nu^{-1}$ denotes the complex dilatation of $(F^\nu)^{-1}$,
we see that $r_\nu:\mathcal M(\mathbb U) \to \mathcal M(\mathbb U)$ is a bijection.
Then, the projection of $r_{\nu}$ to $T_b(\mathbb U)$ by
$\pi:\mathcal M(\mathbb U) \to T_b(\mathbb U)$
is well-defined as $\pi \circ r_{\nu} \circ \pi^{-1}$, which yields a
bijection $R_{[\nu]}:T_b(\mathbb U) \to T_b(\mathbb U)$ for any $[\nu] \in T_b(\mathbb U)$.
This is the right translation of $T_b(\mathbb U)$.

\begin{lemma}\label{auto}
If $F^\nu$ is a bi-Lipschitz diffeomorphism with $\nu \in \mathcal M(\mathbb U)$, then
$r_\nu:\mathcal M(\mathbb U) \to \mathcal M(\mathbb U)$ is a biholomorphic automorphism of $\mathcal M(\mathbb U)$.
For every $[\nu] \in T_b(\mathbb U)$, $R_{[\nu]}:T_b(\mathbb U) \to T_b(\mathbb U)$
is a biholomorphic automorphism of $T_b(\mathbb U)$.
\end{lemma}

\begin{proof}
The argument for showing these statements
is outlined in \cite[Remark 5.1]{SWei}. Here, we add supplemental remarks to this.
To prove that $r_\nu$ is holomorphic, it is enough to show that $r_\nu$ is continuous or locally bounded,
and then verify certain weak holomorphy of $r_\nu$. To prove that $R_{[\nu]}$ is holomorphic,
we consider the Teichm\"uller projection $\pi:\mathcal M(\mathbb U) \to T_b(\mathbb U)$,
which is holomorphic by \cite[Theorem 5.1]{SWei}. A local inverse of $\pi$ is also given there.
Again, we can show that this is continuous or locally bounded first,
and then verify its weak holomorphy. These arguments have been done for a different Teichm\"uller space in \cite{Mat}.
See \cite{WM-6}.
\end{proof}

By Proposition \ref{curve}, we can consider a map $L:{\rm BE} \to {\rm BMO}(\mathbb R)$ defined 
by $L(\gamma)=\log \gamma'$. 
Then, with respect to the complex structure of ${\rm BE}$ given as above,
we see the following:

\begin{theorem}\label{holo}
The map $L:{\rm BE} \to {\rm BMO}(\mathbb R)$ is holomorphic.
\end{theorem}

\begin{proof}
We will prove that $L$ is holomorphic at any point $\gamma=G(\mu_1,\mu_2)|_{\mathbb R}$ in ${\rm BE}$.
Since ${\rm BE}$ can be regarded as a domain of the product 
${\rm BMOA}(\mathbb L) \times {\rm BMOA}(\mathbb U)$ of the Banach spaces, 
the Hartogs theorem for Banach spaces (see \cite[Theorem 14.27]{Ch} and \cite[Theorem 36.8]{Mu}) implies that
we have only to prove that $L$ is separately holomorphic. Thus, 
by fixing $[\mu_1] \in T_b(\mathbb U)$, we will show that 
$\log (G(\mu_1,\mu)|_{\mathbb R})' \in {\rm BMO}(\mathbb R)$ depends holomorphically on $[\mu] \in T_b(\mathbb L)$.
The other case is similarly treated.

By the proof of Proposition \ref{curve}, we have
$$
\log (G(\mu_1,\mu)|_{\mathbb R})' =\log h' \circ f + \log f',
$$
where $f: \mathbb R \to \mathbb R$ is the boundary extension of 
the fixed bi-Lipschitz diffeomorphism $F:\mathbb L \to \mathbb L$
for the Teichm\"uller class $[\mu_1]$,
and $h:\mathbb R \to \mathbb C$ is the restriction of
the quasiconformal homeomorphism $H$ of $\mathbb C$ that is conformal on $\mathbb U$ and
has the complex dilatation $F_* \mu$ on $\mathbb L$. 
Since $H|_{\mathbb U}$ depends on $R_{[\mu_1]}^{-1}([\mu]) \in T_b(\mathbb L)$ and 
$R_{[\mu_1]}^{-1}([\mu])$ is holomorphic on $[\mu]$ by Lemma \ref{auto},
we see that ${\mathcal L}_{{H([\mu])}}=\log (H|_{\mathbb U})'
\in {\rm BMOA}(\mathbb U)$ depends on $[\mu] \in T_b(\mathbb L)$ holomorphically.

By Theorem \ref{131}, we see that the boundary extension $b:{\rm BMOA}(\mathbb U) \to {\rm BMO}(\mathbb R)$ is
a bounded linear operator. Moreover, by Theorem \ref{pullback}, the composition operator
$P_f:{\rm BMO}(\mathbb R) \to {\rm BMO}(\mathbb R)$ induced by $f \in {\rm SQS}$ is also a bounded linear operator.
Therefore,
$$
\log h' \circ f=P_f \circ b ({\mathcal L}_{{H([\mu])}}) \in {\rm BMO}(\mathbb R)
$$
depends on $[\mu] \in T_b(\mathbb L)$ holomorphically, and so does 
$\log (G(\mu_1,\mu)|_{\mathbb R})'$.
\end{proof}

We introduce canonical biholomorphic automorphisms of ${\rm BE}$.
For $\nu \in {\mathcal M}(\mathbb U)$,
the same symbol $\nu$ still denotes the symmetric complex dilatation $\overline{\nu(\bar z)}$ for $z \in \mathbb L$ in 
${\mathcal M}(\mathbb L)$. This also gives the identification of $T_b(\mathbb U)$ and $T_b(\mathbb L)$,
which is often denoted just by $T_b$.
For any $[\nu] \in T_b$, we define the right translation of ${\rm BE}$ by
$$
\widetilde R_{[\nu]}:([\mu_1],[\mu_2]) \mapsto ([\mu_1] \ast [\nu], [\mu_2] \ast [\nu]).
$$
Since $R_{[\nu]}$ is a biholomorphic automorphism of $T_b$ by Lemma \ref{auto},
$\widetilde R_{[\nu]}$ yields a biholomorphic automorphism of ${\rm BE}$.

By the proof of Theorem \ref{holo}, we see that the holomorphic map $L$ can be represented by
the composition of some biholomorphic automorphism $\widetilde R_{[\,\cdot\,]}$ and
the boundary extension of the logarithm of the derivative of some Riemann mapping.
We will find a more explicit relation between $L$ and $\widetilde R_{[\,\cdot\,]}$
in the next section.

Based on the proof of Theorem \ref{holo}, we also define the {\it conformal welding coordinates} of ${\rm BE}$ as follows.
Under the Bers coordinates ${\rm BE} \cong T_b(\mathbb U) \times T_b(\mathbb L)$,
the subspace ${\rm SQS} \subset {\rm BE}$ of all
normalized strongly quasisymmetric homeomorphisms is identified with the diagonal locus
$$
\{([\mu],[\mu]) \in T_b(\mathbb U) \times T_b(\mathbb L) \mid [\mu] \in T_b\},
$$
which is a real-analytic submanifold of ${\rm BE}$.
Let ${\rm RM}$ be the set of all BMO embeddings $\gamma:\mathbb R \to \mathbb C$
of {\it Riemann mapping parametrization}, that is, those
extending conformally to $\mathbb U$.
The subspace ${\rm RM} \subset {\rm BE}$ 
is identified with the second coordinate axis
$$
\{([0],[\mu]) \in T_b(\mathbb U) \times T_b(\mathbb L) \mid [\mu] \in T_b\},
$$
which is a complex-analytic submanifold of ${\rm BE}$. 
We define the projections to these submanifolds
$$
\Pi: {\rm BE} \to {\rm SQS}, \qquad \Phi: {\rm BE} \to {\rm RM}
$$
by $\Pi([\mu_1],[\mu_2])=([\mu_1],[\mu_1])$ and $\Phi([\mu_1],[\mu_2])=([0],[\mu_2] \ast [\mu_1]^{-1})$ 
represented in the Bers coordinates. 
Then, every $\gamma \in {\rm BE}$ 
is decomposed uniquely into $\gamma=\Phi(\gamma) \circ \Pi(\gamma)$.
Clearly, $\Pi$ is real-analytic. 
We will see that $\Phi$ is not continuous later in Proposition \ref{continuity}. 
The biholomorphic automorphism $\widetilde R_{[\nu]}$ of ${\rm BE}$ for $[\nu] \in T_b$
satisfies that $\Phi \circ \widetilde R_{[\nu]}=\Phi$.

The projections $\Pi$ and $\Phi$ define another product structure ${\rm SQS} \times {\rm RM}$ on ${\rm BE}$.
Namely, we have a bijection
$$
(\Pi,\Phi):{\rm BE} \to {\rm SQS} \times {\rm RM}.
$$
However, $(\Pi,\Phi)$ is not a homeomorphism.
Since ${\rm SQS}$ and ${\rm RM}$ are both identified with $T_b$, $(\Pi,\Phi)$ is the
coordinate change of ${\rm BE}$ from the Bers coordinates
to the one we may call the
conformal welding coordinates: 
$$
T_b(\mathbb U) \times T_b(\mathbb L) \to T_b \times T_b: \quad([\mu_1], [\mu_2]) \mapsto ([\mu_1],[\mu_2] \ast [\mu_1]^{-1}).
$$

\section{Parameter change and arc-length parametrization}
To investigate the image $L({\rm BE})$ in ${\rm BMO}(\mathbb R)$,
we prepare canonical biholomorphic auto\-morphisms of ${\rm BMO}(\mathbb R)$ that keep
$L({\rm BE})$ invariant. Under $L$ restricted to the subset of BMO embeddings whose images are chord-arc curves,
these biholomorphic automorphisms correspond to change of parameters for chord-arc curves.

We have assumed that ${\rm BMO}(\mathbb R)$ is the complex Banach space of the equivalence classes of BMO functions modulo 
complex constant functions,
and thus we can regard it as the set of representatives $w$ satisfying the normalization condition
$\int_0^1 e^{w(t)}dt=1$. Let ${\rm BMO}_{\mathbb R}(\mathbb R)$ and
$i{\rm BMO}_{\mathbb R}(\mathbb R)$ denote the real subspaces of 
${\rm BMO}(\mathbb R)$ consisting of all real-valued and purely imaginary-valued functions respectively,
both of which are taken modulo complex constant functions. 
Let $u \in {\rm BMO}_{\mathbb R}^*(\mathbb R)$ and let
$\gamma_{u}:\mathbb R \to \mathbb R$ be the strongly quasisymmetric homeomorphism in ${\rm SQS}$
defined by $\gamma_{u}(x)=\int_0^x e^{u(t)}dt$. Then, the composition operator
$P_{\gamma_u}:{\rm BMO}(\mathbb R) \to {\rm BMO}(\mathbb R)$ is given by $w \mapsto w \circ \gamma_{u}$ 
for $w \in {\rm BMO}(\mathbb R)$,
which is a linear isomorphism of the Banach space ${\rm BMO}(\mathbb R)$ onto itself by Theorem \ref{pullback}. 
Moreover, we define $Q_{u}(w)=P_{\gamma_u}(w)+u$ for $w \in {\rm BMO}(\mathbb R)$, 
which is an affine isomorphism of ${\rm BMO}(\mathbb R)$ onto itself. 
Clearly, $Q_{u}$ preserves the real subspace ${\rm BMO}_{\mathbb R}(\mathbb R)$.

We first show the correspondence of two translations, $\widetilde R_{[\nu]}$ for $\rm BE$ and $Q_u$ for ${\rm BMO}(\mathbb R)$.

\begin{proposition}\label{Lequation}
It holds that 
$$L \circ \widetilde R_{[\nu]}=Q_{L([\nu],[\nu])} \circ L$$
on ${\rm BE}$ for every $[\nu] \in T_b$.
For any $u \in {\rm BMO}_{\mathbb R}^*(\mathbb R)$, $Q_u$ is a biholomorphic automorphism of ${\rm BMO}(\mathbb R)$
that keeps $L({\rm BE})$ invariant.
\end{proposition}

\begin{proof}
For any $\gamma=([\mu_1],[\mu_2]) \in {\rm BE}$ in the Bers coordinates, we have
\begin{equation*}
\begin{split}
L \circ \widetilde R_{[\nu]}([\mu_1],[\mu_2])&=L([\mu_1] \ast [\nu], [\mu_2] \ast [\nu])\\
&=\log \gamma'_{L([\mu_1] \ast [\nu], [\mu_2] \ast [\nu])}=\log (\gamma_{L([\mu_1],[\mu_2])}\circ \gamma_{L([\nu],[\nu])})'\\
&=\log \gamma'_{L([\mu_1],[\mu_2])}\circ \gamma_{L([\nu],[\nu])}+ \log \gamma'_{L([\nu],[\nu])}\\
&=L([\mu_1],[\mu_2]) \circ \gamma_{L([\nu],[\nu])}+L([\nu],[\nu])=Q_{L([\nu],[\nu])} \circ L([\mu_1],[\mu_2])
\end{split}
\end{equation*}
as required. Since $Q_u$ is an affine isomorphism of ${\rm BMO}(\mathbb R)$, this is biholomorphic.
For any $u \in {\rm BMO}_{\mathbb R}^*(\mathbb R)$, we choose the corresponding $[\nu] \in T_b$ under (\ref{trinity}).
Then, $L([\nu],[\nu])=u$. By the above formula, we have
$$
Q_u(L({\rm BE}))=Q_{L([\nu],[\nu])} \circ L({\rm BE})=L \circ \widetilde R_{[\nu]}({\rm BE})
=L({\rm BE}),
$$
and thus $L({\rm BE})$ is invariant under $Q_u$.
\end{proof}

Let $iv \in i{\rm BMO}_{\mathbb R}(\mathbb R) ^{\circ}$ for the subset given as
$$
i{\rm BMO}_{\mathbb R}(\mathbb R) ^{\circ}=i{\rm BMO}_{\mathbb R}(\mathbb R) \cap L({\rm BE}).
$$
Then,
$\gamma_{iv}(x)=\int_0^x e^{iv(t)}dt$ is a BMO embedding
satisfying $|\gamma_{iv}'|=1$. This condition means that $\gamma$ is parametrized by
its arc-length.
By Theorem \ref{Ainfty}, the image of
$\gamma_{iv}$ is a chord-arc curve.
We consider the parameter change of BMO embeddings whose images are chord-arc curves.

\begin{proposition}\label{arclength}
Let $u \in {\rm BMO}_{\mathbb R}^*(\mathbb R)$ and $iv \in i{\rm BMO}_{\mathbb R}(\mathbb R) ^{\circ}$.
Then, $\gamma_{Q_u(iv)}(x)$ is obtained from the BMO embedding $\gamma_{iv}(x')$ of arc-length parametrization
by the change of parameter $x'= \gamma_u(x)$, which is also 
a BMO embedding whose image is a chord-arc curve. 
Hence, an injective map 
$$
J:{\rm BMO}_{\mathbb R}^*(\mathbb R) \times i{\rm BMO}_{\mathbb R}(\mathbb R) ^{\circ} \to 
L({\rm BE}) \subset {\rm BMO}(\mathbb R)
$$
is defined by $J(u,iv)=Q_u(iv)=u+iP_{\gamma_u}(v)$. 
\end{proposition}

\begin{proof}
Since $Q_u(iv)=u+iP_{\gamma_u}(v)=u+iv \circ \gamma_u$, we have
\begin{equation*}
\begin{split}
\gamma_{Q_u(iv)}(x)=\int_0^x e^{u(t)} e^{iv \circ \gamma_u(t)}dt
=\int_0^x \gamma'_u(t) e^{iv \circ \gamma_u(t)}dt
=\int_0^{\gamma_u(x)} e^{iv(s)}ds=\gamma_{iv}(\gamma_u(x))
\end{split}
\end{equation*}
by $s=\gamma_u(t)$. By Proposition \ref{Lequation}, we see that the parameter change of a BMO embedding 
by $\gamma_u \in {\rm SQS}$ is
also a BMO embedding.
\end{proof}

Let ${\rm CA}$ be a subset of ${\rm BE}$ consisting of all normalized BMO embeddings $\gamma:\mathbb R \to \mathbb C$
whose images are chord-arc curves.
By Theorem \ref{Ainfty}, we see that  
\begin{equation}\label{LCA}
L({\rm CA})=L({\rm BE}) \cap \{w \in {\rm BMO}(\mathbb R) \mid {\rm Re}\, w \in {\rm BMO}_{\mathbb R}^*(\mathbb R)\}.
\end{equation}
In particular, both ${\rm BMO}_{\mathbb R}^*(\mathbb R)$ and $i{\rm BMO}_{\mathbb R}(\mathbb R)^\circ$
are contained in $L({\rm CA})$. 

\begin{proposition}\label{BMO*}
The affine isomorphism $Q_u$ for every $u \in {\rm BMO}_{\mathbb R}^*(\mathbb R)$ keeps the convex open subset
${\rm BMO}_{\mathbb R}^*(\mathbb R)$ invariant, and hence it keeps $L(\rm CA)$ invariant.
More explicitly, $Q_u$ maps the affine subspace $u_1+i{\rm BMO}_{\mathbb R}(\mathbb R)$
for any $u_1 \in {\rm BMO}_{\mathbb R}^*(\mathbb R)$ onto $u_2+i{\rm BMO}_{\mathbb R}(\mathbb R)$  
for $u_2=Q_u(u_1)\in {\rm BMO}_{\mathbb R}^*(\mathbb R)$. 
\end{proposition}

\begin{proof}
For $u, u' \in {\rm BMO}_{\mathbb R}^*(\mathbb R)$, we take the corresponding strongly quasisymmetric
homeo\-morphisms $\gamma_u, \gamma_{u'} \in {\rm SQS}$. Then, the indefinite
integral of 
$$
\exp(Q_u(u'))=\exp(u'\circ \gamma_u+u)
$$ 
is $\gamma_{u'} \circ \gamma_{u}$. Since ${\rm SQS}$ is a group under the composition,
we see that $\gamma_{u'} \circ \gamma_{u} \in {\rm SQS}$, and thus 
$Q_u(u') \in {\rm BMO}_{\mathbb R}^*(\mathbb R)$.
\end{proof}

Proposition \ref{arclength} implies that the image of
${\rm BMO}_{\mathbb R}^*(\mathbb R) \times i{\rm BMO}_{\mathbb R}(\mathbb R)^\circ$ by the injection $J$ is
in fact contained in $L({\rm CA})$. Further, we have:

\begin{theorem}\label{bijective}
Every BMO embedding whose image is a chord-arc curve is represented by $\gamma_{Q_u(iv)}$
for some $u \in {\rm BMO}_{\mathbb R}^*(\mathbb R)$ and $iv \in i{\rm BMO}_{\mathbb R}(\mathbb R)^\circ$.
Hence, 
$$
J:{\rm BMO}_{\mathbb R}^*(\mathbb R) \times i{\rm BMO}_{\mathbb R}(\mathbb R)^\circ \to 
L({\rm CA})
$$
is a bijection.
\end{theorem}

\begin{proof}
Let $\gamma_{u+iv'}$ be any BMO embedding for $u+iv' \in L({\rm CA})$. 
Theorem \ref{Ainfty} shows that $u \in {\rm BMO}_{\mathbb R}^*(\mathbb R)$.
Then, by choosing 
$v \in {\rm BMO}_{\mathbb R}(\mathbb R)$ satisfying $P_{\gamma_u}(v)=v'$, we see that $\gamma_{u+iv'}(x)$ is obtained
from $\gamma_{iv}(x')$ by the change of the parameter $x'= \gamma_u(x)$. Hence, 
$\gamma_{u+iv'}=\gamma_{Q_u(iv)}$. This implies that $J$ is surjective.
Combined with Proposition \ref{arclength}, this proves that $J$ is bijective.
\end{proof}

We conclude that $i{\rm BMO}_{\mathbb R}(\mathbb R)^\circ$ is a parameter space of all chord-arc curves
with arc-length parametrizations and ${\rm BMO}_{\mathbb R}^*(\mathbb R)$ is a parameter space
for the change of their parameters. 
Although $J$ gives a bijection onto $L({\rm CA})$, 
we will see later that $J$ is not a homeomorphism (Theorem \ref{mapL}).

\begin{remark}
In \cite[Section 5]{FHS},  
the set of all normalized strongly quasisymmetric embeddings of $\mathbb R$ called 
the {\it extended BMO Teichm\"ul\-ler space} was denoted by $\hat T_b$.
We see that $\hat T_b$ coincides with $\rm CA$ if we use Theorem \ref{0.2} in our paper. 
It was also proved that $L|_{\rm CA}$ is a bijection onto an open subset of ${\rm BMO}(\mathbb R)$ (Theorem 5.2).
Moreover, the product structure of $\hat T_b$ was given (Theorem 5.3),
which is ${\rm CA}={\rm SQS} \times {\rm ICA}$ in our notation explained in Section 7. This
essentially defines the bijection $J$ through $L$.
Under this translation, Problem 5.4 in \cite{FHS} can be understood
as asking whether $J$ is a homeomorphism or not.
\end{remark}

\section{Biholomorphic correspondence on the space of chord-arc curves}
We will show that $L$ is biholomorphic on ${\rm CA}$. 
We note that the space ${\rm SQS} \subset {\rm BE}$ of all normalized strongly quasisymmetric homeomorphisms
$\gamma:\mathbb R \to \mathbb R$ is contained in ${\rm CA}$ because the image of $\gamma$ is $\mathbb R$,
which is of course a chord-arc curve. Hence, we in particular obtain that there is a biholomorphic correspondence
between the neighborhood of ${\rm SQS}$ in $\rm BE$ and that of ${\rm BMO}^*_{\mathbb R}(\mathbb R)$ in
${\rm BMO}(\mathbb R)$ under $L$.

\begin{theorem}\label{biholo}
The holomorphic map $L:{\rm BE} \to {\rm BMO}(\mathbb R)$ restricted to ${\rm CA} \subset {\rm BE}$
is a biholomorphic homeomorphism onto its image. This in particular implies that
${\rm CA}$ is open in ${\rm BE}$ and $L({\rm CA})$ is open in ${\rm BMO}(\mathbb R)$.
\end{theorem}

\begin{proof}
It suffices to show that $L$ has a local holomorphic inverse at 
any point $w \in L({\rm CA}) \subset {\rm BMO}(\mathbb R)$. 
We see that if $w=iv \in i{\rm BMO}_{\mathbb R}(\mathbb R)^\circ$,
then there is a neighborhood $V_{iv} \subset L({\rm CA})$ of $iv$ and
a holomorphic map $\Psi_{iv}^*:V_{iv} \to {\rm BE}$ such that $L \circ \Psi_{iv}^*={\rm id}|_{V_{iv}}$.
Indeed, it was proved in \cite[Proposition 4.13]{Se} that there is a BMO quasiconformal homeomorphism $G_w:\mathbb C \to \mathbb C$
with complex dilatation $\mu_w$ 
for every $w$ in some neighborhood $V_{iv}$ so that 
the map $\Lambda_{iv}:V_{iv} \to \mathcal M(\mathbb U) \times \mathcal M(\mathbb L)$ defined by this correspondence
is continuous at $w=iv$. Since ${\rm BMO}_{\mathbb R}^*(\mathbb R)$ is open by Proposition \ref{convex},
we may choose a smaller neighborhood $V_{iv}$ so that ${\rm Re}\, w \in {\rm BMO}_{\mathbb R}^*(\mathbb R)$ if $w \in V_{iv}$.
In this case, the image of the BMO embedding $G_w|_{\mathbb R}$ is a chord-arc curve by Theorem \ref{Ainfty}.
Then, the local inverse $\Psi_{iv}^*$ of $L$ is given by $\widetilde \pi \circ \Lambda_{iv}$, and $V_{iv} \subset L({\rm CA})$
by (\ref{LCA}).

The holomorphy of $\Lambda_{iv}$ is proved as follows.
We may assume that $\Lambda_{iv}$ is bounded on $V_{iv}$ by the continuity of $\Lambda_{iv}$ at $iv$. 
Then, it suffices to show that
$\Lambda_{iv}$ is a G\^ateaux holomorphic function
(see \cite[Theorem 14.9]{Ch}, \cite[Theorem 36.5]{Mu}). However, because $\mu_w$ can be represented
explicitly in terms of $w \in V_{iv}$, we can verify this weak holomorphy as in \cite[Theorem 6.1]{SW}
though the norm used there is different from ours.
A similar argument using the norm of $\mathcal M(\mathbb U)$ is contained in \cite{WM-6}.

If $w=u+iv'$ is an arbitrary point in $L({\rm CA})$, then by Theorem \ref{bijective},
we can find $iv \in i{\rm BMO}_{\mathbb R}(\mathbb R)^\circ$ satisfying $Q_u(iv)=u+iv'$. 
Since $Q_u$ is a biholomorphic automorphism of ${\rm BMO}(\mathbb R)$ keeping $L({\rm CA})$ invariant by 
Proposition \ref{BMO*},
we see that $\widetilde R_{[\nu]}\circ \Psi_{iv}^* \circ Q_u^{-1}$ is holomorphic on $Q_u(V_{iv}) \subset L({\rm CA})$
for $[\nu] \in T_b$ corresponding to $u \in {\rm BMO}_{\mathbb R}^*(\mathbb R)$. Then, Proposition \ref{Lequation} shows that
$$
L \circ \widetilde R_{[\nu]} \circ \Psi_{iv}^* \circ Q_u^{-1}=Q_{L([\nu],[\nu])} \circ L \circ \Psi_{iv}^* \circ Q_u^{-1}
=Q_{L([\nu],[\nu])} \circ Q_u^{-1}=\rm id
$$
on $Q_u(V_{iv})$. Hence, $\widetilde R_{[\nu]}\circ \Psi_{iv}^* \circ Q_u^{-1}$ is 
a local holomorphic inverse of $L$ on $Q_u(V_{iv})$.
\end{proof}

Let ${\rm ICA} \subset {\rm CA}$ denote the subset of all normalized 
chord-arc curves
with arc-length parametrization. Namely,
$$
{\rm ICA}=L^{-1}(i{\rm BMO}_{\mathbb R}(\mathbb R)^{\circ}).
$$
As $i{\rm BMO}_{\mathbb R}(\mathbb R)^{\circ}$ is a real-analytic submanifold of 
the domain $L({\rm CA})$ in the complex Banach space ${\rm BMO}(\mathbb R)$,
${\rm ICA}$ is a real-analytic submanifold of the complex manifold ${\rm CA}$.
Similarly, since ${\rm SQS}=L^{-1}({\rm BMO}_{\mathbb R}^*(\mathbb R))$, 
this is also a real-analytic submanifold of ${\rm CA}$.

\begin{corollary}\label{real-analytic}
$(1)$
${\rm SQS}$ is a real-analytic submanifold of ${\rm CA}$ that is
the diagonal of $T_b(\mathbb U) \times T_b(\mathbb L)$ in the Bers coordinates, and
$L|_{{\rm SQS}}$ is a real-analytic homeomorphism onto ${\rm BMO}_{\mathbb R}^*(\mathbb R)$ whose inverse 
is also real-analytic.
$(2)$
${\rm ICA}$ is a real-analytic submanifold of ${\rm CA}$, and
$L|_{\rm ICA}$ 
is a real-analytic homeomorphism onto $i{\rm BMO}_{\mathbb R}(\mathbb R)^{\circ}$ whose inverse 
is also real-analytic.
\end{corollary}

Part (1) of the above corollary extends Proposition \ref{topequiv} concerning $\Psi=L|_{\rm SQS}$.
This shows that ${\rm SQS}$ is equipped with both the complex-analytic structure of $T_b$
and the real-analytic structure of ${\rm BMO}_{\mathbb R}^*(\mathbb R)$, which are real-analytically equivalent.
Introducing the real-analytic structure to $T_b$ from the real-analytic submanifold ${\rm SQS}$
on the diagonal of ${\rm BE} \cong T_b(\mathbb U) \times T_b(\mathbb L)$ is usual if we know the theory of
quasifuchsian deformation spaces.

\begin{remark}
We consider a problem of the connectivity of $\rm CA$ (see \cite[p.614]{AZ} and \cite[Problem 5.5]{FHS}).
Since $L$ is a homeomorphism on $\rm CA$, we may transfer this problem to $L({\rm CA}) \subset {\rm BMO}(\mathbb R)$.
By Theorem \ref{bijective}, $J:{\rm BMO}_{\mathbb R}^*(\mathbb R) \times i{\rm BMO}_{\mathbb R}(\mathbb R)^\circ \to 
L({\rm CA})$ is bijective. Even though $J$ is not continuous, for each fixed $u \in {\rm BMO}_{\mathbb R}^*(\mathbb R)$,
$J(u, \cdot)=Q_u=u+P_{\gamma_u}$ is continuous on $i{\rm BMO}_{\mathbb R}(\mathbb R)^\circ$.
Thus, the problem on the connectivity for $L({\rm CA})$ is reduced to that for $i{\rm BMO}_{\mathbb R}(\mathbb R) ^{\circ}
=i{\rm BMO}_{\mathbb R}(\mathbb R) \cap L({\rm BE})$. Since ${\rm BE} \cong T_b(\mathbb U) \times  T_b(\mathbb L)$,
$L({\rm BE})$ is connected.
Then, the problem is to show the connectivity of the intersection of
$L({\rm BE})$ with the linear subspace
$i{\rm BMO}_{\mathbb R}(\mathbb R)$.
If this were not connected, then distinct components would be joined by a path in $L({\rm BE})$ taking a detour 
whose real parts escape from
${\rm BMO}_{\mathbb R}^*(\mathbb R)$.
\end{remark}

We add one more property to
the biholomorphic mapping 
$L:{\rm CA} \to {\rm BMO}(\mathbb R)$ restricted to ${\rm SQS}$,
which is concerning the correspondence of bounded subsets. 
Here, the boundedness in ${\rm SQS}$ is considered regarding
a metric structure of ${\rm SQS} \cong T_b$.
The invariant metric provided for $T_b$ is the Carleson metric (see \cite{WM-3}), and
let $d_c$ denote the Carleson distance in $T_b$. 

The composition operator $P_f$ is a bounded linear operator as in Theorem \ref{pullback}.
Concerning uniform boundedness of this operator norm,
we see the following proposition.
This is stated in \cite[p.18]{CM}. We provide a proof for it here.

\begin{proposition}\label{uniformPh}
For any $f \in {\rm SQS}$,
let $P_f:{\rm BMO}(\mathbb R) \to {\rm BMO}(\mathbb R)$ be the bounded linear operator defined by
$w \mapsto w \circ f$ for every $w \in {\rm BMO}(\mathbb R)$. Then, there exist 
constants $\tau_0>0$ and $C_0>0$ such that
the operator norm of $P_f$ satisfies $\Vert P_f \Vert \leq C_0$
for every $f \in {\rm SQS}$ with $\Vert \log f' \Vert_* \leq \tau_0$.
\end{proposition}

\begin{proof}
Lemma \ref{nearid} implies that there exist $\delta_0 >0$ and $\tau_0>0$
such that if $\Vert \log h' \Vert_* \leq \delta_0$ and
$\Vert \log f' \Vert_* \leq \tau_0$, then $\Vert \log (h\circ f)' \Vert_* \leq 1$.
We may choose $\tau_0$ so that $\tau_0 \leq 1$. Then, for any $u=\log h' \in {\rm BMO}_{\mathbb R}(\mathbb R)$
with $\Vert u \Vert_* \leq \delta_0$ and $f \in {\rm SQS}$ with $\Vert \log f'\Vert_* \leq \tau_0$,
we have
$$
\Vert P_f(u) \Vert_* \leq \Vert \log(h \circ f)'\Vert_*+\Vert \log f'\Vert_*\leq 2.
$$
For $w=u+iv \in {\rm BMO}(\mathbb R)$ with $\Vert w \Vert_* = \delta_0$, we apply this estimate to
$\Vert P_f(u) \Vert_*$ and $\Vert P_f(v) \Vert_*$; we obtain that
if $\Vert w \Vert_* = \delta_0$ and $\Vert \log f'\Vert_* \leq \tau_0$, then 
$\Vert P_f(w) \Vert_* \leq 4$. 
This completes the proof by setting $C_0=4/\delta_0$.
\end{proof}

This can be extended to the local uniform boundedness of the operator norm as follows. 

\begin{lemma}\label{general}
For any $f \in {\rm SQS}$,
the operator norm $\Vert P_f \Vert$ of the composition operator $P_f:{\rm BMO}(\mathbb R) \to {\rm BMO}(\mathbb R)$
depends only on the Carleson distance $d_c(f,\rm id)$.
\end{lemma}

\begin{proof}
For the constant $\tau_0$ in Proposition \ref{uniformPh}, we choose a constant $r_0>0$ such that
if $f \in {\rm SQS}$ satisfies $d_c(f,{\rm id}) \leq r_0$ then $\Vert \log f' \Vert_* \leq \tau_0$.
Any element $f \in {\rm SQS}$ can be joined to $\rm id$ by a curve in ${\rm SQS}$
with its length arbitrary close to $d_c(f,{\rm id})$. We choose the minimal number of consecutive points
$$
{\rm id}=f_0, f_1, \ldots, f_n=f
$$
on the curve such that $d_c(f_i,f_{i-1}) < r_0$ for any $i=1,\ldots,n$. Then, the number $n$ is determined by
$d_c(f,\rm id)$, and the invariance of $d_c$ under the right translation implies that the composition
$f_i \circ f_{i-1}^{-1}$ satisfies 
$d_c(f_i \circ f_{i-1}^{-1},{\rm id}) <r_0$, and hence $\Vert \log (f_i \circ f_{i-1}^{-1})'\Vert_* \leq \tau_0$.
By decomposing $f$ into these $n$ mappings, we have
$$
P_f=P_{f_1 \circ f_0^{-1}} \circ P_{f_2 \circ f_1^{-1}} \circ \cdots \circ P_{f_n \circ f_{n-1}^{-1}}.
$$
Then, Proposition \ref{uniformPh} shows that $\Vert P_f \Vert \leq C_0^n$.
\end{proof}

\begin{remark}
It is known that $\Vert P_f \Vert$ can be estimated only by the constants $\alpha$ and $K$ for 
the strongly quasisymmetric condition (\ref{alphaK}) of $f \in {\rm SQS}$. See \cite[Example 2.3]{Got}.
\end{remark}

\begin{theorem}\label{b-b}
For any bounded subset $W \subset {\rm SQS}$, the image $L(W)$ is
also bounded in ${\rm BMO}_{\mathbb R}(\mathbb R)$. In more details, if $f \in {\rm SQS}$ is
within distance $r$ from $\rm id$ in the Carleson distance $d_c$, then $\Vert \log f' \Vert_*$
is bounded by a constant depending only on $r$. 
\end{theorem}

\begin{proof}
For $f \in {\rm SQS}$ with $d_c(f,{\rm id}) \leq r$, we decompose $f$ as in the proof of
Lemma \ref{general}. Then, by this lemma, we have
\begin{align*}
\Vert \log f'\Vert_* 
&=\Vert \log ((f_n \circ f_{n-1}^{-1}) \circ (f_{n-1} \circ f_{n-2}^{-1}) \circ \cdots \circ (f_{1} \circ f_{0}^{-1}) )'\Vert_*\\
& \leq \Vert P_{f_1 \circ f_0^{-1}} \circ P_{f_2 \circ f_1^{-1}} \circ \cdots \circ P_{f_{n-1} \circ f_{n-2}^{-1}}(\log (f_n \circ f_{n-1}^{-1})')\Vert_*\\
&\quad + \cdots +\Vert \log(f_{1} \circ f_{0}^{-1})' \Vert_*\\
& \leq C_0^{n-1} \tau_0+ C_0^{n-2} \tau_0 + \cdots +\tau_0.
\end{align*}
Since $n$ depends only on $r$, the statement is proved.
\end{proof}

\section{Arc-length and Riemann mapping parametrizations}

Any chord-arc curve $\Gamma$ admits two canonical parametrizations. One is arc-length parametrization and the other is
Riemann mapping parametrization, which is
given by the boundary extension of the Riemann mapping from $\mathbb U$ to the left domain bounded by $\Gamma$.
In this section, we consider the correspondence of these parametrizations.

Let ${\rm RM}^\circ={\rm RM} \cap {\rm CA}$. 
This is the subset of ${\rm BE}$ consisting of all chord-arc curves with Riemann mapping parametrizations.
Since ${\rm CA}$ is an open subset of ${\rm BE}$ by Theorem \ref{biholo},
${\rm RM}^\circ$ is an open subset of the complex-analytic submanifold ${\rm RM}$ of ${\rm BE}$.
Under the identification ${\rm RM} \cong T_b$, this corresponds to an open subset $T_c$ of $T_b$
investigated in \cite{WM-3}.
Every $\gamma \in {\rm CA}$ 
is decomposed uniquely into $\gamma=h \circ f$, where $h=\Phi(\gamma) \in {\rm RM}^\circ$ and
$f=\Pi(\gamma) \in {\rm SQS}$. Conversely, the parameter change of $h \in {\rm RM}^\circ$ by $f \in {\rm SQS}$
gives $\gamma=h \circ f \in {\rm CA}$. In this sense, ${\rm CA}$ is represented by
the product ${\rm SQS} \times {\rm RM}^\circ$, which is in fact induced by
the conformal welding coordinates $T_b \times T_b$ of ${\rm BE}$.

Similarly, ${\rm CA}$ is represented by the product ${\rm SQS} \times {\rm ICA}$ (see \cite[Theorem 5.3]{FHS}). 
Namely, every $\gamma \in {\rm CA}$ is decomposed uniquely into $\gamma=\gamma_0 \circ f$, and vice versa,
which can be understood as the parameter change of $\gamma_0 \in {\rm ICA}$ by $f \in {\rm SQS}$.
Under the map $J^{-1} \circ L$, this decomposition corresponds to the product structure
${\rm BMO}_{\mathbb R}^*(\mathbb R) \times i{\rm BMO}_{\mathbb R}(\mathbb R)^{\circ}$.

Having two product structures ${\rm SQS} \times {\rm RM}^\circ$ and 
${\rm SQS} \times {\rm ICA}$ on ${\rm CA}$, we compare the arc-length parametrizations ${\rm ICA}$ 
with the Riemann mapping para\-metri\-zations ${\rm RM}^\circ$. 
Each fiber of the projection $\Phi:{\rm CA} \to {\rm RM}^\circ$ consists of a family of normalized BMO embeddings
with the same image, 
and hence they have the same arc-length parametrization. 
This observation leads to the following.

\begin{proposition}\label{fibers}
There is a bijection between ${\rm ICA}$ and ${\rm RM}^\circ$ 
keeping the images of the corresponding embeddings the same. 
This bijection is nothing but $\Phi|_{\rm ICA}:{\rm ICA} \to {\rm RM}^\circ$.
\end{proposition}

We consider the other projection $\Pi$ restricted to ${\rm ICA}$,
which has been studied with great interest in the literature.
For any $\gamma_0 \in {\rm ICA}$, $\Pi(\gamma_0) \in {\rm SQS}$ is defined by
the strongly quasisymmetric homeomorphism of $\mathbb R$ inducing
the parameter change from $\gamma_0$ to $\Phi(\gamma_0) \in {\rm RM}^\circ$. 
The following result was proved by Coifman and Meyer \cite[Theorem 1]{CM}
by operator theoretical arguments. See also \cite[Section 6]{Se} and \cite{Wu}.
In our formulation, this statement itself is already self-evident.

\begin{theorem}\label{CM}
The map $\Pi|_{\rm ICA}:{\rm ICA} \to {\rm SQS}$ is real-analytic. Hence,
$$
\lambda=L \circ \Pi \circ L^{-1}|_{i{\rm BMO}_{\mathbb R}(\mathbb R)^\circ}:i{\rm BMO}_{\mathbb R}(\mathbb R)^\circ \to 
{\rm BMO}_{\mathbb R}^*(\mathbb R)
$$ 
is also real-analytic.
\end{theorem}

\begin{proof}
By Corollary \ref{real-analytic}, ${\rm ICA}$ is a real-analytic submanifold of ${\rm CA}$.
Hence, the restriction $\Pi|_{\rm ICA}$ of the projection $\Pi:{\rm CA} \to {\rm SQS}$ is real-analytic.
Since $L$ is biholomorphic by Theorem \ref{biholo}, the conjugate map $\lambda$ is real-analytic.
\end{proof}

However, what is essential in properties of $\Pi|_{\rm ICA}$ is that it is injective and
the inverse map is also real-analytic. See \cite[Theorem 5]{SeB}.
In fact, the original arguments for Theorem \ref{CM} and their generalization
were obtained by investigating this inverse map.

\begin{theorem}\label{inverse}
The map $\Pi|_{\rm ICA}$ is injective, its image is
an open subset of ${\rm SQS}$, and the inverse $(\Pi|_{\rm ICA})^{-1}$ is real-analytic.
Equivalently, $\lambda=L \circ \Pi \circ L^{-1}|_{i{\rm BMO}_{\mathbb R}(\mathbb R)^\circ}$ is injective,
its image is an open subset of ${\rm BMO}_{\mathbb R}^*(\mathbb R)$, and the inverse $\lambda^{-1}$ is
real-analytic.
\end{theorem}

We will only prove the injectivity of $\lambda$ and explain the existence of the derivative $d_0(\lambda^{-1})$ at the origin
which is invertible. Then, we in particular see that some neighborhoods at the origin of $\rm ICA$ and ${\rm SQS}$ 
correspond homeomorphically by $\Pi|_{\rm ICA}$. This claim will be used in the next section.

Every $\gamma_0 \in {\rm ICA}$ is decomposed uniquely into $\gamma_0=h \circ f$ for $h \in {\rm RM}^\circ$ and $f \in {\rm SQS}$.
Taking the logarithm of the derivative of this equation, we have
\begin{equation*}\label{logder}
\log \gamma_0'= \log h' \circ f+\log f'.
\end{equation*}
Since $\log \gamma_0'=iv$ is purely imaginary and $\log f'$ is real, the real and the imaginary parts of this equation
become
\begin{equation}\label{ReIm}
0={\rm Re} \log h' \circ f +\log f'\quad {\rm and} \quad v={\rm Im} \log h' \circ f.
\end{equation}
Moreover, since $\log h'$ is the boundary extension of the holomorphic function $\log H'$ for the Riemann mapping $H$ on $\mathbb U$,
${\rm Re}\log h'$ and ${\rm Im}\log h'$ are related by the Hilbert transformation $T$ on $\mathbb R$:
\begin{equation}\label{Hilbert}
{\rm Im}\log h'=T({\rm Re}\log h').
\end{equation}
Then, the combination of (\ref{ReIm}) and (\ref{Hilbert}) yields that
\begin{equation}\label{injective}
-P_f\circ T \circ P_f^{-1}(\log f')=v.
\end{equation}
This shows that $v$ is determined by $f$ and thus $\lambda:\log \gamma_0' \mapsto \log f'$ is injective.
 
Let $u=\log f' \in {\rm BMO}_{\mathbb R}^*(\mathbb R)$. Then, equation (\ref{injective}) also gives
$$
\lambda^{-1}(u)=-iP_f\circ T \circ P_f^{-1}(u).
$$
Here, we assume a fact that the conjugate $P_f\circ T \circ P_f^{-1}$ by the composition operator $P_f$
tends to the Hilbert transformation $T$ in the operator norm as $f \to \rm id$ in ${\rm SQS} \cong T_b$
(in spite of the fact that $P_f \nrightarrow I$ in Corollary \ref{last}).
This was proved in \cite[Theorem 1]{CM0} in a more general form. In fact, $P_f\circ T \circ P_f^{-1}$ was shown to be
real-analytic as an operator valued function, from which Theorem \ref{inverse} follows.
A conceptional explanation for this fact is in \cite[p.86]{CS}.
Indeed, $P_f\circ T \circ P_f^{-1}$ can be extended from the function on $u=\log f' \in {\rm BMO}_{\mathbb R}^*(\mathbb R)$
to that on $w=\log \gamma' \in L({\rm CA})$ as the Cauchy integral operator on a chord-arc curve. Then, local boundedness and
weak holomorphy of $P_\gamma\circ T \circ P_\gamma^{-1}$ imply that this function is holomorphic.
The local boundedness of $P_f$ is shown in Proposition \ref{uniformPh}.
In addition, we can find a proof for an estimate $\Vert P_f\circ T \circ P_f^{-1} -T \Vert \lesssim \Vert u \Vert_*$
when $\Vert u \Vert_*$ is small
in \cite[p.206]{Se1}.

By $P_f\circ T \circ P_f^{-1} \to T$ as $f \to \rm id$, we obtain that $\lambda^{-1}$ is differentiable at the origin and
$d_0(\lambda^{-1})=-iT$. Since the Hilbert transformation $T$ is an isomorphism of ${\rm BMO}_{\mathbb R}(\mathbb R)$
onto itself, we see that $d_0(\lambda^{-1})$ is invertible and in fact $d_0\lambda=-iT$. By the inverse function theorem,
there are some neighborhoods at the origin of $i{\rm BMO}_{\mathbb R}(\mathbb R)^\circ$ and ${\rm BMO}_{\mathbb R}^*(\mathbb R)$
that are real-analytically equivalent under $\lambda$.

\section{Discontinuity of the correspondence to Riemann mappings}

In this section, we prove that the bijection $\Phi|_{\rm ICA}:{\rm ICA} \to {\rm RM}^\circ$ is not continuous. This means that
the correspondence of Riemann mappings to the arc-length parametrizations of chord-arc curves
is not continuous. To this end,
we first consider a problem on the topological group structure of $T_b \cong {\rm SQS}$.

It is well-known that the universal Teichm\"uller space $T$,
which can be identified with the group ${\rm QS}$ of all 
normalized quasisymmetric homeomorphisms of $\mathbb R$ onto itself, 
does not constitute a topological group. An example of discontinuity of the map
$([\mu_1],[\mu_2]) \mapsto [\mu_2] \ast [\mu_1]^{-1}$ is given in \cite{Le}.
We adjust this example slightly to 
see that $T_b$ is not a topological group either.
On the contrary, Proposition \ref{nearid} implies that $T_b$ is a partial topological group.
In this case, \cite[Lemma 1.2]{GS} proved that the subgroup consisting of all elements $[\nu] \in T_b$
such that $[\nu] \circ [\mu_n] \circ [\nu]^{-1} \to [0]$ for any sequence $[\mu_n] \in T_b$ converging to $[0]$
is a closed topological group, which is called the characteristic topological subgroup of $T_b$.
We denote this subgroup by ${\rm char}(T_b)$.

\begin{proposition}\label{continuity}
The map $\Phi:{\rm BE} \to {\rm RM}$ is not continuous at $([0],[\nu]) \in {\rm BE}$
with some $[\nu] \neq [0]$ in 
any small neighborhood of the origin of $T_b$, but continuous at a point
$([\mu_1],[\mu_2]) \in {\rm BE}$ if $[\mu_2] \ast [\mu_1]^{-1}$ belongs to ${\rm char}(T_b)$.
\end{proposition}

\begin{proof}
We modify the example in \cite[Theorem III.3.3]{Le}.
For $k>0$,
let $f_k: \mathbb R \to \mathbb R$ be a normalized quasisymmetric homeomorphism defined by
$$
f_k(x)=
\left\{
\begin{array}{ll}
x & (0 \leq x)\\
\frac{k}{k+1}x & (-\frac{k+1}{k} < x <0)\\
x+\frac{1}{k} & (x \leq -\frac{k+1}{k})\quad.
\end{array}
\right.
$$
Since $\Vert \log f_k' \Vert_* \to 0$ as $k \to \infty$, we have $f_k \in {\rm SQS}$ if $k>0$ is sufficiently
large. Under the identification $T_b \cong {\rm SQS}$, we take $[\nu_k] \in T_b$ that corresponds to
$f_k$. Since $[\nu_k] \to [0]$ in $T_b$ as $k \to \infty$, we can choose and fix some $k>0$ 
so that $[\nu]=[\nu_k] \neq [0]$ is arbitrarily close to $[0]$. Set also $f=f_k$. 

Next, for every $n \in \mathbb N$, let
$\ell_n: \mathbb R \to \mathbb R$ be a normalized quasisymmetric homeomorphism defined by
$$
\ell_n(x)=
\left\{
\begin{array}{ll}
x & (x \geq 0)\\
\frac{n}{n+1}x & (x <0)
\quad.
\end{array}
\right.
$$
Similarly to the above, since $\Vert \log \ell_n' \Vert_* \to 0$ as $n \to \infty$, 
the sequence $[\varepsilon_n] \in T_b$ corresponding to $\ell_n$ tends to $[0]$ as $n \to \infty$.
Then, $([\varepsilon_n],[\nu])$ converges to $([0],[\nu])$.
On the contrary, we see that $\Phi([\varepsilon_n],[\nu])=([0],[\nu] \ast [\varepsilon_n]^{-1})$ does not converge to
$\Phi([0],[\nu])=([0],[\nu])$. Indeed, a simple computation yields that
$$
\log (f \circ \ell_n^{-1} \circ f^{-1})'(x)=
\left\{
\begin{array}{ll}
\log \left(1+\frac{1}{n} \right) & (-\frac{n}{n+1} \leq x<0)\\
\log \left(1+\frac{1}{n} \right)+\log \left(1+\frac{1}{k} \right) & (-1 \leq x <-\frac{n}{n+1})
\quad ,
\end{array}
\right.
$$
and hence
$$
\Vert \log(f \circ \ell_n^{-1} \circ f^{-1})'\Vert_* \geq \frac{1}{2}\log\left(1+\frac{1}{k} \right).
$$
This implies that $f \circ \ell_n^{-1}$ does not converge to $f$ as $n \to \infty$ in ${\rm SQS}$.
Alternatively,
the proof of \cite[Theorem III.3.3]{Le} asserts that 
$[\nu] \ast [\varepsilon_n]^{-1}$ does not converge to $[\nu]$ in the universal Teichm\"uller space $T$. 
Since the topology of $T_b$ is stronger than
that of $T$, $[\nu] \ast [\varepsilon_n]^{-1}$ does not converge to $[\nu]$ in $T_b$ either.
This proves that $\Phi$ is not continuous on any small neighborhood of the origin
of ${\rm BE}$.

For the continuity at $([\mu_1],[\mu_2])$ when $[\mu_2] \ast [\mu_1]^{-1}$ belongs to ${\rm char}(T_b)$,
we take any sequences $[\varepsilon_n]$ and $[\varepsilon'_n]$ converging to $[0]$, and
show that $\Phi([\varepsilon_n] \ast [\mu_1],[\varepsilon'_n] \ast[\mu_2]) \to \Phi([\mu_1],[\mu_2])$
as $n \to \infty$. This is equivalent to showing that
$$
[\varepsilon'_n] \ast ([\mu_2] \ast [\mu_1]^{-1}) \ast [\varepsilon_n]^{-1}\ast([\mu_2] \ast [\mu_1]^{-1})^{-1} \to [0].
$$
By the definition of ${\rm char}(T_b)$, this is satisfied.
\end{proof}

\begin{remark}
The VMO Teichm\"uller space $T_v$ can be defined as the set of all normalized strongly symmetric homeomorphisms of
the unit circle onto itself, and this is a closed subspace of $T_b$ (see \cite{SWei}).
The subset $T_c$ of $T_b$ consisting of all elements corresponding to chord-arc curves contains $T_v$ (see \cite{WM-3}).
It was shown in \cite{Wei} that all strongly symmetric homeomorphisms on the unit circle
constitute a topological group, and
from its proof, we see that $T_v$ is contained in ${\rm char}(T_b)$.
We expect that ${\rm char}(T_b)=T_v$ but it seems that this is not known yet.
\end{remark}

Proposition \ref{nearid} implies that $\Phi$ is continuous at the origin of ${\rm BE}$, and 
so is $\Phi|_{\rm ICA}:{\rm ICA} \to {\rm RM}^\circ$. We also consider the continuity of the inverse map
$(\Phi|_{\rm ICA})^{-1}$ at the origin.
This follows from the continuity of $J^{-1}$ at the origin.

\begin{lemma}\label{origin}
$J^{-1}:L({\rm CA}) \to {\rm BMO}_{\mathbb R}^*(\mathbb R) \times i{\rm BMO}_{\mathbb R}(\mathbb R)^\circ$
is continuous at $0$. As a consequence,
$(\Phi|_{\rm ICA})^{-1}:{\rm RM}^\circ \to {\rm ICA}$ is continuous at $([0],[0])$.
\end{lemma}

\begin{proof}
Since $J^{-1}(u+iv)=(u,Q_u^{-1}(u+iv))$ and $Q_{u}^{-1}=L \circ \widetilde R_{[\nu]}^{-1} \circ L^{-1}$
with the correspondence between $u \in {\rm BMO}_{\mathbb R}^*(\mathbb R)$ and $[\nu] \in T_b$,
the continuity of $J^{-1}$ at $0$ follows from the continuity of $R_{[\nu]}^{-1}([\mu])=[\mu] \ast [\nu]^{-1}$
at $([0],[0])$ by Proposition \ref{nearid}. Since
$$
(\Phi|_{\rm ICA})^{-1}=L^{-1}\circ J \circ p \circ J^{-1}\circ L
$$
where $p:{\rm BMO}_{\mathbb R}^*(\mathbb R) \times i{\rm BMO}_{\mathbb R}(\mathbb R)^\circ \to i{\rm BMO}_{\mathbb R}(\mathbb R)^\circ$
is the projection to the second factor, this is continuous at $([0],[0])$.
\end{proof}

We are ready to prove the main result, 
which solves a conjecture of
Katznelson, Nag and Sullivan \cite[p.303]{KNS}
by showing that 
the dependence of the Riemann mapping on the arc-length parametrization is not continuous.
See also \cite[Section 1]{SW}.

\begin{theorem}\label{main}
$\Phi|_{\rm ICA}:{\rm ICA} \to {\rm RM}^\circ$ is not continuous.
\end{theorem}

\begin{proof}
By Proposition \ref{continuity}, we can find a point $([0],[\nu]) \in {\rm RM}^\circ$
arbitrarily close to $([0],[0])$ at which $\Phi$ is not continuous.
Let $(\Phi|_{\rm ICA})^{-1}([0],[\nu])=([\mu_1],[\mu_2])$, which is in ${\rm ICA}$.
Since $(\Phi|_{\rm ICA})^{-1}$ is continuous at $([0],[0])$ by Lemma \ref{origin},
we may assume that $([\mu_1],[\mu_2])$ is sufficiently close to $([0],[0])$.

By Theorems \ref{CM} and \ref{inverse}, there are neighborhoods
$U$ and $V$ of the origin in ${\rm ICA}$ and ${\rm SQS}$ respectively such that
$\Pi|_{U}:U \to V$ is a homeomorphism preserving the origin.
We can take $([\mu_1],[\mu_2])$ so that it is in $U$. Then, $\Pi([\mu_1],[\mu_2])=([\mu_1],[\mu_1]) \in V$.
We choose a sequence $[\varepsilon_n] \in T_b$ as in the proof of Proposition \ref{continuity},
and consider $([\varepsilon_n] \ast [\mu_1],[\varepsilon_n] \ast[\mu_1]) \in V$. This sequence is mapped to
$([\varepsilon_n] \ast [\mu_1],[\varepsilon'_n] \ast[\mu_2]) \in U$ by $(\Pi|_U)^{-1}$, where $[\varepsilon'_n]$
also tends to $[0]$ as $n \to \infty$ by the continuity of $(\Pi|_U)^{-1}$.
Thus, we obtain a sequence $([\varepsilon_n] \ast [\mu_1],[\varepsilon'_n] \ast[\mu_2])$ in ${\rm ICA}$
converging to $([\mu_1],[\mu_2])$. 

The image of this sequence under $\Phi$ is
$$
([0],[\varepsilon'_n] \ast[\mu_2] \ast [\mu_1]^{-1} \ast [\varepsilon_n] ^{-1})
=([0],[\varepsilon'_n] \ast[\nu] \ast [\varepsilon_n] ^{-1}) \in {\rm RM}^\circ.
$$
As we have seen in the proof of Proposition \ref{continuity}, $[\nu] \ast [\varepsilon_n] ^{-1}$
does not converge to $[\nu]$ in $T_b$, and neither does $[\varepsilon'_n] \ast[\nu] \ast [\varepsilon_n] ^{-1}$.
This implies that
$\Phi([\varepsilon_n] \ast [\mu_1],[\varepsilon'_n] \ast[\mu_2])$ does not converge to
$\Phi([\mu_1],[\mu_2])=([0],[\nu])$. Hence, $\Phi|_{\rm ICA}$ is not continuous.
\end{proof}

In fact, the mapping in question was given on the space of BMO functions in \cite{KNS}.
Let $\Pi^*:{\rm BE} \to {\rm SQS}$ be the projection defined by $([\mu_1],[\mu_2]) \mapsto ([\mu_2],[\mu_2])$
in parallel to $\Pi$. The identification ${\rm RM} \cong T_b \cong {\rm SQS}$ is realized by
$\Pi^*|_{\rm RM}:{\rm RM} \to {\rm SQS}$. Then,
$$
\Pi^* \circ \Phi|_{\rm ICA}:{\rm ICA} \to {\rm SQS}
$$
is the correspondence of conformal welding homeomorphisms to arc-length para\-metri\-zations.
Its conjugation by $L$, that is,
$$
\rho=L \circ \Pi^* \circ \Phi \circ L^{-1}|_{i{\rm BMO}_{\mathbb R}(\mathbb R)^\circ} 
:i{\rm BMO}_{\mathbb R}(\mathbb R)^\circ \to {\rm BMO}_{\mathbb R}^*(\mathbb R)
$$
is the one we deal with.  

\begin{corollary}\label{another}
$\rho:i{\rm BMO}_{\mathbb R}(\mathbb R)^\circ \to {\rm BMO}_{\mathbb R}^*(\mathbb R)$ is not continuous.
\end{corollary}

\begin{proof} 
This follows from Theorem \ref{main} and the facts that $L$ and $\Pi^*|_{\rm RM}$ are homeomorphisms.
\end{proof}

We also have the following properties on the bijection $J$. As mentioned before in Section 5,
this answers the problem raised in \cite[Problem 5.4]{FHS} negatively.

\begin{theorem}\label{mapL}
The bijection
$$
J:{\rm BMO}_{\mathbb R}^*(\mathbb R) \times i{\rm BMO}_{\mathbb R}(\mathbb R)^\circ \to 
L({\rm CA}) \subset {\rm BMO}(\mathbb R)
$$
defined by $J(u,iv)=Q_u(iv)=u+iP_{\gamma_u}(v)$
is not continuous but locally bounded. 
\end{theorem}

\begin{proof}
All the properties of $J$ stem from those of the composition operator $P_{\gamma_u}$.
For a given $u_0 \in {\rm BMO}_{\mathbb R}^*(\mathbb R)$, 
we consider any $u \in {\rm BMO}_{\mathbb R}^*(\mathbb R)$ that is sufficiently close to $u_0$.
For the normalized strongly quasisymmetric homeomorphisms $\gamma_u$ and $\gamma_{u_0}$ in
${\rm SQS}$, we set $f=\gamma_u \circ \gamma_{u_0}^{-1}$, which also belongs to ${\rm SQS}$.
Then, $\Vert \log f' \Vert_* \to 0$ as $u \to u_0$.
By Proposition \ref{uniformPh}, if $\Vert \log f' \Vert_* \leq \tau_0$, then $\Vert P_f \Vert \leq C_0$.
Hence, the operator norm of $P_{\gamma_u}=P_{\gamma_{u_0}} \circ P_f$ is bounded by $C_0 \Vert P_{\gamma_{u_0}} \Vert$
for any $u$ in some neighborhood of $u_0$. This implies that $J$ is locally bounded.

For $iv_0 \in i{\rm BMO}_{\mathbb R}(\mathbb R)^\circ$ whose BMO norm is small so that $e^{v_0}$ is an $A_\infty$-weight, 
we can take the normalized strongly quasisymmetric homeomorphism $\gamma_{v_0}:\mathbb R \to \mathbb R$ such that
$\log \gamma_{v_0}'=v_0$. Then,
\begin{equation}\label{realv}
J(u,iv_0)=u+ i \log \gamma_{v_0}'\circ \gamma_u=i \log(\gamma_{v_0} \circ \gamma_u)'+(1-i)u.
\end{equation}
We may choose $\gamma_{v_0} \in {\rm SQS}$ so that it corresponds to $[\nu] \in T_b$ in the proof of
Theorem \ref{continuity}. We may also choose $\gamma_{u_n} \in {\rm SQS}$
for each $n \in \mathbb N$ so that $\gamma_{u_n} \circ \gamma_{u_0}^{-1}$ corresponds to 
$[\varepsilon_n] \in T_b$. 
Then, $[\varepsilon_n] \to [0]$ but $[\nu] \ast [\varepsilon_n]^{-1} \nrightarrow [\nu]$ as $n \to \infty$.
This implies that $\gamma_{u_n} \to \gamma_{u_0}$ but $\gamma_{v_0} \circ \gamma_{u_n} \nrightarrow \gamma_{v_0} \circ \gamma_{u_0}$,
namely, $u_n \to u_0$ but $\log (\gamma_{v_0} \circ \gamma_{u_n})' \nrightarrow \log (\gamma_{v_0} \circ \gamma_{u_0})'$.
By equation (\ref{realv}), this shows that $J$ is not continuous at $(u_0,iv_0)$.
\end{proof}

This also implies discontinuous properties of the composition operator 
$P_f:{\rm BMO}(\mathbb R) \to {\rm BMO}(\mathbb R)$ for $f \in {\rm SQS}$.

\begin{corollary}\label{last}
There are $v_0 \neq 0$ in ${\rm BMO}(\mathbb R)$ with an arbitrarily small norm and
a sequence $f_n$ in ${\rm SQS}$ converging to $\rm id$ such that
$\Vert P_{f_n}(v_0)-v_0 \Vert_*$ does not converge to $0$ as $n \to \infty$.
In particular, the operator norm $\Vert P_{f_n}-I \Vert$ does not necessarily converge to $0$ 
when $f_n \to {\rm id}$.
\end{corollary}

\begin{proof}
For $u_0=0$, choose $f_n=\gamma_{u_n}$ and $v_0$ as in the proof of Theorem \ref{mapL}. Then,
$P_{f_n}(v_0)=\log (\gamma_{v_0} \circ \gamma_{u_n})'-\log \gamma_{u_n}'$ does not converge to
$v_0=\log \gamma_{v_0}'$.
\end{proof}

\end{document}